\documentclass{amsart}

\usepackage{amscd}
\usepackage{amsmath}
\usepackage[toc,page]{appendix}
\usepackage{lmodern}
\usepackage[shortlabels]{enumitem}
\usepackage{caption}
\usepackage[pdfpagelabels=true]{hyperref}
\usepackage{amssymb,tikz,xcolor,url}
\usetikzlibrary{calc}

\newtheorem{thm}{Theorem}[section] 
 
\newtheorem{lem}[thm]{Lemma} 
\newtheorem{prop}[thm]{Proposition}
 
\theoremstyle{remark}
\newtheorem{rem}[thm]{Remark} 
\theoremstyle{definition}
\newtheorem{defin}[thm]{Definition} 
\newtheorem{example}[thm]{Example} 

\newcommand{\Wloc}{W_{\mbox{\tiny loc}}}
\newcommand{\inner}{\partial^{\,\mbox{\tiny inner}}}
\newcommand{\side}{\partial^{\,\mbox{\tiny side}}}

 \title[Heat kernel estimates]{Heat kernel estimates on manifolds with ends  with mixed boundary condition}
\author{Emily Dautenhahn}
\thanks{Partially supported by NSF grants DMS-1645643 and DMS-2054593.}
\address{Department of Mathematics and Statistics, Murray State University}
\author{Laurent Saloff-Coste}
\thanks{Partially supported by NSF grants DMS-1707589, DMS-2054593, and DMS- 234386.}
\address{Department of Mathematics, Cornell University}

\begin{document}

\begin{abstract}
    We obtain two-sided heat kernel estimates for Riemannian manifolds with ends with mixed boundary condition, provided that the heat kernels for the ends are well understood. These results extend previous results of Grigor'yan and Saloff-Coste by allowing for Dirichlet (or mixed) boundary condition. The proof requires the construction of a global harmonic function which is then used in the $h$-transform technique. 
\end{abstract}

\maketitle

\section{Introduction} \subsection{Motivation} In \cite{GS1}, Alexander Grigor'yan and the second author initiated the study of two-sided heat kernel estimates on weighted complete  Riemannian manifolds with finitely many nice ends, $M=M_1\# \cdots \#M_k$. The components  $M_i$ of this connected sum are, themselves,  assumed to be weighted complete Riemannian manifolds. The main assumption is that, on each $M_i$,  the heat kernel $p_{M_i}(t,x,y)$,  is well understood in the sense that it satisfies a classical-looking two-sided Gaussian estimate, uniformly at all times and locations. Equivalently (\cite{Gri,PSHDuke,Asp}),  the volume functions of these manifolds, $M_i$, $1\le i\le k$, are uniformly doubling at all scales and locations \textbf{and} their geodesic balls satisfy a Neumann-type Poincar\'e inequality, uniformly at all all scales and locations.  These are very strong hypotheses, and, in certain cases, additional more technical hypotheses are needed. The results of \cite{GS1} are sharp two-sided estimates on the heat kernel of $M.$ The most basic case illustrating these results is when
$M_i=\mathbb R^N$ for some $N,$ and, more generally, $M_i= \mathbb R^{n_i}\times \mathbb S^{N-n_i}$, for some $N$ and $n_i$, $1\le n_i\le N$.
These basic cases were new and already plenty challenging at the time \cite{GS1} was published. They are richer than they appear if one takes into consideration the variation afforded by the weight functions.  In addition, the results hold without change when the term ``complete Riemannian manifold''
is interpreted in the context of manifolds with boundary. Complete, then, means metrically complete, and the heat equations and heat kernels on $M$ and on the $M_i$, $1\le i\le k$, are all taken with Neumann boundary condition. So, for instance, the results of \cite{GS1,GS5}
 cover the  solid three-dimensional body in Figure \ref{fig1}. (This figure created by A. Grigor'yan appears in \cite{GS5}.)

\begin{figure}
  \includegraphics[width=0.3\linewidth]{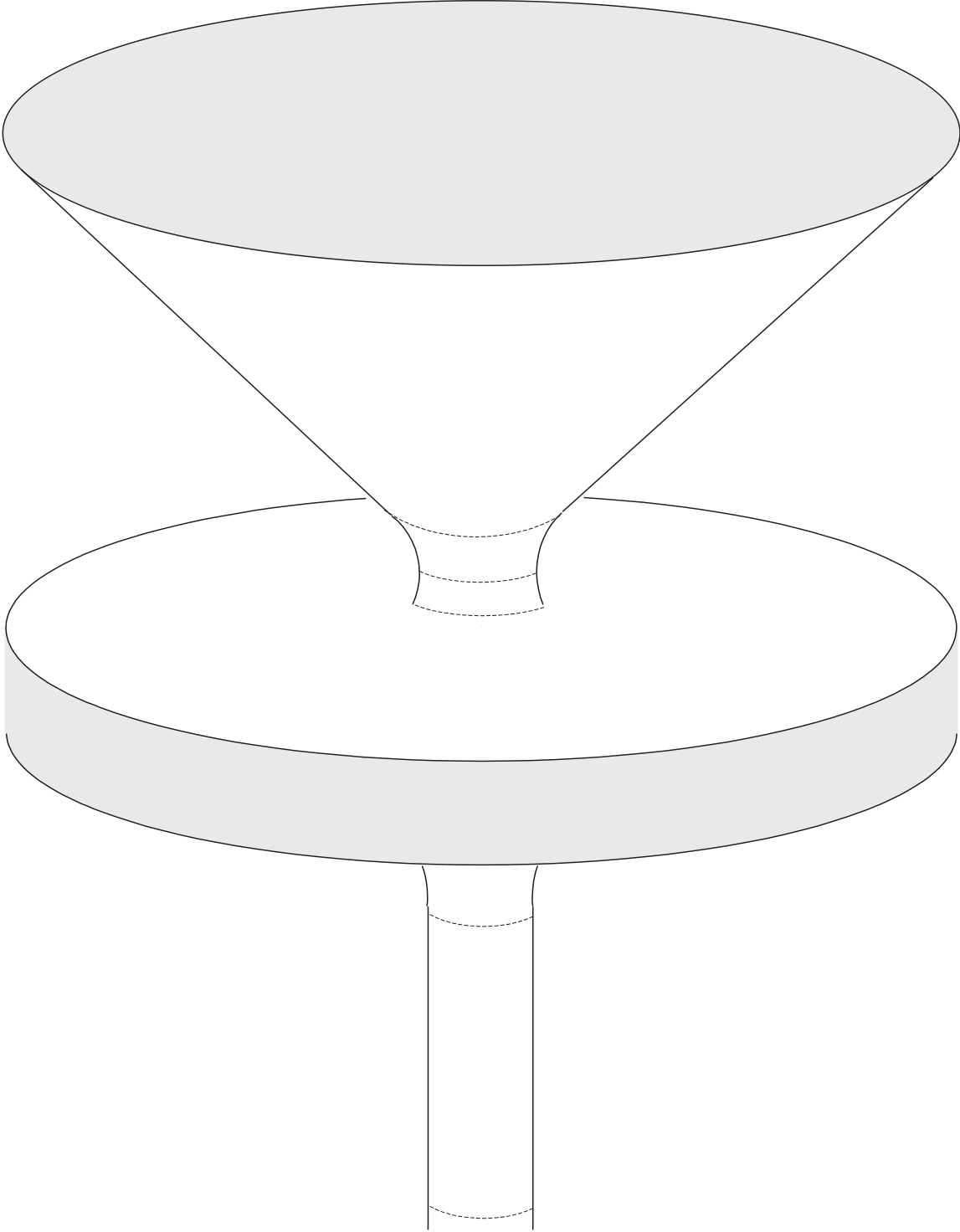}
  \caption{A solid subset of $\mathbb R^3$: a complete manifold with boundary.}
  \label{fig1}
\end{figure}

The aim of the present work is to initiate the study of the case when the heat equation on the complete manifold $M$ (with boundary) above is taken with mixed boundary condition: Neumann on some part of the boundary and Dirichlet on the rest of the boundary. (Of course, restrictive assumptions will be made on the nature of the set on which Dirichlet boundary condition holds.) Here, as usual,  Neumann boundary condition refers to the requirement that the normal derivative of the solution vanishes at the boundary, whereas Dirichlet boundary condition refers to the vanishing of the solution itself at the boundary. Even the simplest possible instances of this problem, such as the planar domain depicted in Figure \ref{fig:cones1}, present interesting challenges.

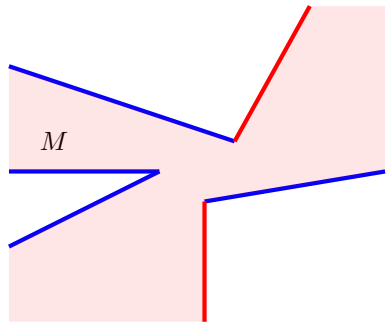
\begin{figure}\begin{center}\begin{tikzpicture}[scale=.2] 
\draw[blue,ultra thick] (0,5)--(10,5) ; \draw[blue,ultra thick] (0,0) -- (10,5);
\draw[blue, ultra thick] (0,12)--(15,7); \draw[red, ultra thick] (15,7)--(20,16); \node at ((3,7) {$M$};

\draw[blue, ultra thick] (25,5)--(13,3); \draw[red, ultra thick] (13,3)--(13,-5) ;
\path [fill, red, opacity=.1]  (0,5)--(10,5) --(0,0) --(0,-5)--(13,-5)--(13,3)--(25,5)--(25,16)--(20,16)--(15,7)--(0,12)--(0,5);

\end{tikzpicture}\end{center}\caption{Sketch of a planar, unbounded, complete manifold $M$ (light red) with boundary $\delta M$ (blue and dark red)  with three conic ends. Dirichlet boundary is depicted in blue, Neumann boundary in red. Corners should be rounded so that $M$ is really a (smooth) manifold with boundary, though actually it does not matter; see Appendix \ref{corners}.}\label{fig:cones1}\end{figure}

\begin{figure}\begin{center}\begin{tikzpicture}[scale=.2] 
\draw[red,ultra thick] (0,5)--(10,5) ;\draw[red,ultra thick] (0,0) -- (10,5);
\draw[blue, ultra thick] (0,12)--(15,7); \draw[red, ultra thick] (15,7)--(20,16); \node at ((3,7) {$M$};

\draw[blue, ultra thick] (25,5)--(13,3); \draw[red, ultra thick] (13,3)--(13,-5) ;
\path [fill, red, opacity=.1]  (0,5)--(10,5) --(0,0) --(0,-5)--(13,-5)--(13,3)--(25,5)--(25,16)--(20,16)--(15,7)--(0,12)--(0,5);
\end{tikzpicture}\end{center}\caption{Same $M$ as in Figure \ref{fig:cones1}, but with different boundary conditions.}\label{fig:cones2}\end{figure}
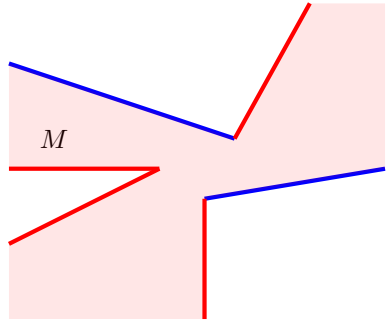

Describing the behavior of the heat kernel in the domain depicted in Figures \ref{fig:cones1} and \ref{fig:cones2} (with the given boundary conditions) will require the introduction of a fair bit of notation. Ultimately, in our main result of Theorem \ref{main_thm}, we give upper and lower bounds, valid for all time $t>0$ and pairs $(x,y)\in M$,  which are essentially ``matching bounds" in the sense used  widely in the literature on heat kernel bounds. For the purpose of this introduction, we focus on the following particular case: Fix a point $o$ in $M$ not on the Dirichlet boundary. What is the behavior of $p(t,o,o)$ as $t$ tends to infinity when $M$ is the domain depicted in Figure \ref{fig:cones1} with the given boundary conditions?

To answer this question, starting from upper left and continuing counter-clockwise, denote by $M_1 ,M_2, M_3$ the three cones whose connected sum is $M$.  Note that $M_1$ carries Dirichlet boundary condition on both sides whereas  $M_2$  and $M_3$ carry Dirichlet boundary condition on one side and Neumann on the other.
Let $\alpha_i$ be the apertures of $M_i$, $1\le i\le 3$. We will show that, because each $\alpha_i$ is positive,  there are constants $0<c_o\le C_o<+\infty$ such that, for all $t>1$,
\[ c_ot^{-a} \le  p(t,o,o)\le C_ot^{-a} \]
with 
\[a= 1+\min\left\{\frac{\pi}{\alpha_1},\frac{\pi}{2\alpha_2},\frac{\pi}{2\alpha_3}\right\}.\] 

Now consider the case where there is a cone of positive aperture carrying Neumann boundary condition on both sides and (at least) one other cone carrying Dirichlet boundary condition on at least one side as in Figure \ref{fig:cones2}. We will show that this situation leads to the behavior
 \[ c_o(t\log^2 t)^{-1} \le  p(t,o,o)\le C_o(t\log^2 t)^{-1}.\]

To give yet another variation, consider the domain depicted in Figure \ref{fig:cones3}, which has an end that is a cone of aperture zero with Neumann condition on both sides.

\begin{figure}[h]\begin{center}\begin{tikzpicture}[scale=.2] 
\draw[red,ultra thick] (10,-5) -- (10,8); \draw[red, ultra thick] (13,6)--(13,-5) ;
\draw[red, ultra thick] (0,12)--(25,12);
\node at (3,7) {$M$};
\draw[blue, ultra thick] (13, 6)--(25, 0);
\draw[blue, ultra thick] (10, 8)--(0, -3);
\path[fill, red, opacity=.1] (0,-3)--(10,8)--(10,-5)--(13,-5)--(13,6)--(25,0)--(25, 12)--(0,12)--(0,-3);

\end{tikzpicture}\end{center}\caption{An $M$ with an end that is a cone of aperture zero.}\label{fig:cones3}\end{figure}

In this case,
\[ c_ot^{-3/2} \le  p(t,o,o)\le C_o t^{-3/2} .\]
The case of Dirichlet boundary condition along at least one side of a cone of aperture zero cannot be treated by the techniques of this paper.

\subsection{Description of the Method} These results will be obtained via a general method based on a combination of the basic ideas developed in \cite{GS3,GS2,GS4,GS5,GS6} and \cite{GyS} (the results in \cite{GS5} make heavy use of those in \cite{GS3,GS2,GS4,GS6}).
Reference \cite{GS5} provides a very general line of attack to reconstruct heat kernel estimates on  a connected sum from heat kernel estimates on and related knowledge of the parts forming that sum. Reference \cite{GyS} provides the ideas that make the technique of \cite{GS5} applicable to the case when Dirichlet boundary condition is present. Namely, after an appropriate $h$-transform (also known as Doob's transform after Joseph Doob), the Dirichlet condition disappears, and one can apply the technique of \cite{GS5} straightforwardly, even though the set-up is not quite that of \cite{GS5}. (Appendix \ref{Harnack} contains the relevant adaptations of the results from \cite{GS5}.) Further connections with earlier results are described in Appendix \ref{earlier_results}.

The layout of the paper is as follows. Section \ref{notation} introduces the specific objects and setting under consideration. Section \ref{construct} constructs a harmonic function with special properties to be used as the key function $h$ in the $h$-transform technique. Section \ref{prove_thm} then implements the $h$-transform technique to obtain the desired heat kernel estimates, and Section \ref{exs} applies the main theorem, Theorem \ref{main_thm}, to various examples. The paper concludes with several appendices, which deal with a slightly more general hypothesis than that found in the main portion of the paper and extend the main result to manifolds with simple corners. The appendices also serve to remind the reader of useful definitions and constructions. Furthermore, they contain restatements of several previous results that are crucial to the present paper. As such the reader is frequently referred to the relevant appendix. 

\section{Set-up and notation}\label{notation}

\subsection{The underlying complete manifold \texorpdfstring{$M$}{M}} We start with a smooth manifold with boundary, $(M,\delta M)$, equipped with a Riemannian structure $g$ and a positive, smooth weight $\sigma:M\to (0,+\infty)$.  It will sometimes be useful to set $M^\bullet=M\setminus\delta M$.
We let $d$ be the geodesic distance on $ (M,g)$ and assume that $(M,d)$ is a complete metric space. We call $M$ a weighted, complete Riemannian manifold with boundary (by definition, a manifold is connected).  Hence $M$ comes equipped with a number of additional objects we briefly describe.
\begin{itemize}
\item The Riemannian measure is denoted by $dx,$ and its weighted version $\mu$ is given by $\mu(dx)=\sigma(x) dx.$ We view $(M,d,\mu)$ as our main metric measure space. 
\item Geodesic balls in $M$, which are denoted by $B_M(x,r)$, $x\in M$, $r>0$. The $\mu$ volume of $B_M(x,r)$ is $V(x,r):=\mu(B_M(x,r)).$
\item  The gradient  $\nabla f$ defined on smooth functions by 
\[df|_x(X)=g_x(\nabla f (x), X)\]
for any tangent vector $X$ at $x\in M$.
\item The divergence $\mbox{div} X =\mbox{div}_\mu X$ defined on smooth vector fields by 
\[ \int_M \mbox{div}(X)  f \, d\mu=-\int_M g(X,\nabla f) \, d\mu\] for any smooth compactly supported function $f$ on $M$.
\item The Laplace operator $\Delta=\Delta_{\mu}$ defined on smooth functions on $M^\bullet=M\setminus \delta M$ by $\Delta f=\mbox{div} (\nabla f)$. 
\end{itemize}

\begin{defin}[The Sobolev space $W^1_0(V)$]\label{localSobolev}  The (local) Sobolev space  $\Wloc(M^\bullet)$ is the space of distributions on $M^\bullet$ which can be represented locally by an $L^2$ function and whose first partial derivatives in any precompact local chart of $M^\bullet$ can also be represented by $L^2$ functions. For any open set $U^\bullet \subset M^\bullet,$ we may define $\Wloc(U^\bullet)$ in the same way by replacing $M^\bullet$ with $U^\bullet.$  For any open subset $V\subset M$, the Sobolev space $W_0(V)=W^1_0(V)$ is the subspace of $L^2(V)=L^2(V,\mu|_V)$ obtained by closing the space of smooth compactly supported functions on $V$, $\mathcal C_c^\infty(V)$, under the norm $\left(\int_V|f|^2 d\mu+\int_V|\nabla f|^2 d\mu\right)^{1/2}$. 
\end{defin}

\begin{defin}[Heat equation on $M$] The heat semigroup $P^M_t$ 
is the semigroup associated with the Dirichlet form $(W_0(M),\int_Mg(\nabla f,\nabla f) d\mu)$. It is given on $L^2(M)$  by
\[P^M_t f(x)=\int_M p_M(t,x,y)f(y) \, d\mu(y), \;\;t>0, \ x\in M,\]
where the heat kernel $p_M$, viewed as a function of $t$ and $x$, satisfies the heat equation $(\partial_t-\Delta)p(t,x,y)=0$ on $M\setminus \delta M$ with Neumann boundary condition along the boundary $\delta M$ and the initial condition $p_M(0,x,\cdot)=\delta_x(\cdot)$ 
(when the  distribution/smooth function pairing is given by the extension of  $\langle \phi,\psi\rangle = \int \phi\psi \, d\mu$). The infinitesimal generator associated with this Dirichlet form will be referred to as $\Delta_M.$ 
\end{defin}

\subsection{Our main objects of study}
The complete manifold $M$ and its heat kernel are not the main objects of interest in the present work. Instead, we consider an open subset $\Omega$ of $M$  such that  the closed set $M\setminus\Omega$ is a subset of $\delta M$. Hence, the topological boundary of $\Omega$ in $M$ is $\partial \Omega=M\setminus \Omega$.

\begin{figure}[h]\begin{center}\begin{tikzpicture}[scale=.15] 
\draw[red,ultra thick] (0,5)--(10,5) ; \draw[red,ultra thick] (0,0) -- (10,5);
\draw[red, ultra thick] (0,12)--(15,7); \draw[red, ultra thick] (15,7)--(20,16); \node at ((3,7) {$M$};

\draw[red, ultra thick] (25,5)--(13,3); \draw[red, ultra thick] (13,3)--(13,-5) ;
\path [fill, red, opacity=.4]  (0,5)--(10,5) --(0,0) --(0,-5)--(13,-5)--(13,3)--(25,5)--(25,16)--(20,16)--(15,7)--(0,12)--(0,5);

\draw[blue,ultra thick] (35,5)--(45,5) ; \draw[blue,ultra thick] (35,0) -- (45,5);
\draw[blue, ultra thick] (35,12)--(50,7); \draw[red, ultra thick] (50,7)--(55,16); \node at ((38,7) {$\Omega$};

\draw[blue, ultra thick] (60,5)--(48,3); \draw[red, ultra thick] (48,3)--(48,-5) ;
\path [fill, red, opacity=.1]  (35,5)--(45,5) --(35,0) --(35,-5)--(48,-5)--(48,3)--(60,5)--(60,16)--(55,16)--(50,7)--(35,12)--(35,5); 

\end{tikzpicture}\end{center}\caption{Sketch (corners should be rounded) of the complete manifold $M$ (dark red)  and its submanifold $\Omega$ (light red and red boundary) with ``Dirichlet boundary'' $\partial \Omega\subseteq \delta M$, not part of $\Omega$, highlighted in blue.}\label{fig:cones4}\end{figure}
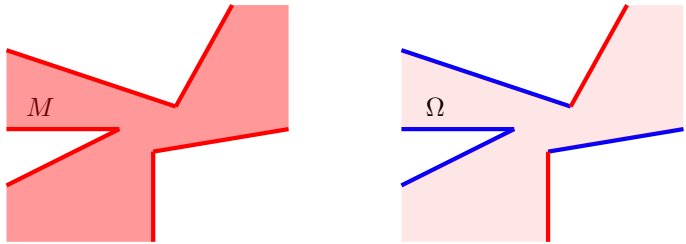

We can view $\Omega$ as a manifold with boundary $\delta\Omega=\delta M\cap \Omega,$ but it is not metrically complete if $\partial\Omega \neq \emptyset$. The metric completion of $\Omega$ is (isometric to) $M$.  We will use the following notation:

\begin{itemize}
\item Geodesic balls in $\Omega$ are denoted by $B(x,r)=B_\Omega(x,r)$. The $\mu$-volume of $B(x,r)$ is $V(x,r)=\mu(B(x,r))=\mu(B_M(x,r))$. By abuse 
of language and notation, if $x\in \partial \Omega$, we write $B_\Omega(x,r)=B_M(x,r)\cap \Omega$. 
\item The heat semigroup $P_t=P^\Omega_t$ and its kernel $p(t,x,y)=p_\Omega(t,x,y),\linebreak (t,x,y)\in (0,+\infty)\times \Omega\times \Omega$, are related on $L^2(\Omega)$  by
\[P_t f(x)=\int_\Omega p (t,x,y)f(y)\, d\mu(y), \;\;t>0, \ x\in M,\]
and are associated with the Dirichlet form $(W_0(\Omega),\int_\Omega g(\nabla f,\nabla f)d\mu),$ which has infinitesimal generator $\Delta_\Omega.$ By definition, the heat kernel $p=p_\Omega$, viewed as a function of $t$ and $x,$ satisfies the heat equation $(\partial_t-\Delta)p(t,x,y)=0$ on $\Omega\setminus \delta \Omega$ with Neumann boundary condition along the boundary $\delta \Omega$ and Dirichlet boundary condition (in the weak sense) along $\partial \Omega$. 
\end{itemize}

\noindent \emph{Condition (*):} Throughout, we make the simplifying assumptions that the closed set $\partial \Omega \subseteq \delta M$ has countably many connected components, each of which is a smooth codimension $1$ manifold with boundary, and that any point in $M$ has a neighborhood in $M$ containing at most finitely many connected components of $\partial \Omega.$

\subsection{Finitely many nice ends}\label{nice-ends}

We now describe the main additional hypotheses we make on the global geometric structure of $M$ (and hence $\Omega$). Namely, we assume that $M$ is the connected sum of $k$ complete Riemannian manifolds with boundary \linebreak $(M_1,\delta M_1),\dots,(M_k,\delta M_k)$, which we write as
\[M=M_1\# M_2\#\cdots\#M_k.\]
With $\sqcup$ denoting disjoint union, this means that 
$$M= K \sqcup \left(E_1\sqcup \dots \sqcup E_k\right), $$
where $K$ is a compact subset of $M$ with the property that $M\setminus K$ has $k$ connected components $E_1,\dots, E_k$ and each $E_i$ is isometric to a connected subset of $M_i$ with compact complement $K_i$ (hence, $M_i=K_i\sqcup E_i$). The explicit decomposition 
$M= K \sqcup \left(E_1\sqcup \dots\sqcup E_k\right) $ is, of course, not unique, and we will assume this decomposition possesses additional nice properties. We assume that the metric closure of each $E_i$ is, itself,  a manifold with boundary. This is a somewhat constraining hypothesis, but it has the advantage of simplifying exposition by restricting our attention to smooth manifolds with boundary. The weight $\sigma$ on $M$ is assumed to be compatible with a weight $\sigma_i$ on  each $M_i$ in the sense that $\sigma|_{E_i}=\sigma_i$. 

Next, we consider an open set  $\Omega \subset M$ with $M\setminus \Omega\subseteq \delta M$ and set
$$U_i= \Omega\cap E_i, \;\;i=1,\dots,k.$$
The open sets $U_i$ are important to us. Each $U_i$ is  a  weighted Riemannian manifold with boundary $\delta U_i=\delta M\cap U_i$ and whose topological boundary in $M$, denoted by $\partial U_i$, is the union of its ``lateral" boundary or ``side" boundary
$\side U_i=\overline{E_i}\cap \partial \Omega$ and its ``inner"  boundary  $\inner U_i=\partial E_i$. The inner boundary $\inner U_i=\partial E_i$ is also a subset of $K$. It is compact with finitely many connected components, which are co-dimension $1$ submanifolds with boundary. The union  $\side U_i \cup  \inner U_i$ is not necessarily disjoint, but the intersection $\side U_i \cap  \inner U_i$ is of co-dimension at least $2$.
We make the following strong hypotheses:

\begin{itemize} 
\item
[(H1)]Each $(M_i,\sigma_i)$ is a Harnack manifold (Definition \ref{HarnackManifold}). Equivalently, each $M_i$ is doubling and the Poincar\'e inequality holds, both uniformly (see Definitions \ref{VD} and \ref{PI}).
\item[(H2)] Each $U_i$ is uniform in $M_i$ (Definition \ref{uniform}).
\end{itemize}
We now collect a list of important known consequences of these hypotheses for future reference.
\begin{itemize}
\item [(C1)] The condition that each $U_i$ is uniform in $M_i$ implies that $U_i$ satisfies (RCA) from Definition \ref{RCA}. This can be seen directly from the definitions. 
\item [(C2)] The condition that each $U_i$ is uniform in the Harnack manifold $M_i$ implies that the elliptic boundary Harnack inequality holds uniformly in $U_i$ (Definition \ref{EBHI}). Further details on the elliptic boundary Harnack inequality and situations in which it holds can be found in \cite{BM, GyS, JLLSC1} and the references therein.
\item [(C3)]The condition that each $U_i$ is uniform in the Harnack manifold $M_i$ implies that $U_i$ admits a harmonic profile $u_i$, that is, a positive harmonic function vanishing along $\partial U_i$ (see Definition \ref{def-vanish}). This profile is unique up to a positive multiplicative constant. This follows, e.g., from Theorem 4.1 of \cite{GyS}. (See also \cite{LiTam} and references therein for discussion regarding existence of harmonic functions on ends on complete Riemannian manifolds without boundary relating to curvature conditions.) 

\item[(C4)] The weighted Riemannian manifold $(U_i, \sigma u_i^2)$ is a Harnack manifold. This is given by Theorem 5.9 of \cite{GyS}.
\end{itemize}

\section{Construction of a Profile for \texorpdfstring{$\Omega$}{Omega}}\label{construct}

\subsection{Harmonic Profiles for \texorpdfstring{$\Omega$}{Omega}}

Throughout this section, we assume all hypotheses given in Section \ref{nice-ends}. In order to use the technique of \cite{GS5}, we need to apply an appropriate $h$-transform. The effect of this will be to ``hide" the Dirichlet boundary and take us to the setting of a connected sum of Harnack manifolds. The goal of this section is to construct a positive global harmonic function in $\Omega$ (Definition \ref{def-globalharm}) that grows at least as fast in each end as the profile for that end; we may refer to this function as a profile for $\Omega.$ While the profiles for the ends $U_i, \ 1 \leq i \leq k,$ are unique up to constant multiples (see (C3)), this is in general not the case for $\Omega,$ even with the additional restriction on the growth of the function. Our main result in this section is that $\Omega$ always possesses a harmonic function of this type.

\begin{thm}\label{thm-profile}
Assuming $\partial \Omega \not = \emptyset,$ there exists a positive harmonic function $h$ on $\Omega,$ vanishing along $\partial \Omega,$ such that $h \geq cu_i$ for some constant $0<c < +\infty,$ where $u_i$ denotes the profile for $U_i,\ 1 \leq i \leq k,$ as in (C3). 
\end{thm}

If $\partial \Omega = \emptyset,$ then $\Omega = M$ is complete and this case is covered by \cite{GS5}, provided $\Omega$ is non-parabolic. In the case of \emph{complete} Riemannian manifolds with ends, see work of Li and Tam \cite{LiTam} and subsequent work of Sung, Tam, and Wang \cite{STW} for discussions of harmonic functions and the relation with certain curvature conditions.

Theorem \ref{thm-profile} is proved by using the profiles $u_i,\ 1 \leq i \leq k,$ to construct a global harmonic function on $\Omega,$  which we then show satisfies all of the desired further properties. The overall construction follows that given in \cite{STW}, but various technicalities arise due to the presence of the Dirichlet boundary $\partial \Omega$. Before giving the proof, we gather some additional consequences of our hypotheses.

\subsection{Behavior of Green Functions}

The proof of the theorem relies heavily on the behavior of the Green function $G$ of $\Omega,$ which exists since $\partial \Omega \not = \emptyset$ (see Appendix \ref{parabolic}). The behavior of $G$ is closely related to the behavior of the Green functions for the ends $U_i, \  G_{U_i},\ 1 \leq i\leq k,$ which exist since all ends $U_i$ are non-parabolic as $\inner U_i \not = \emptyset.$ In turn, what we can say about the behavior of $G_{U_i}$ relies on the strong hypotheses we require of the ends, as well as whether the underlying manifolds $M_i$ are parabolic or non-parabolic. 

\begin{defin} We say that a continuous function $f$ on $\Omega$ tends to zero at infinity in an end $U_i$ if, for all $\varepsilon >0,$ there exists a compact set $K_\varepsilon \subset M$ such that $|f(x)| \leq \varepsilon$ for all points $x \in U_i \setminus K_\varepsilon.$ 

Similarly, we say $f$ tends to zero at infinity in $\Omega$ if, for all $\varepsilon >0,$ there exists a compact set $K_\varepsilon \subset M$ such that $|f(x)| \leq \varepsilon$ for all points $x \in \Omega \setminus K_\varepsilon.$  
\end{defin}

\begin{defin}
Fix points $o_i \in U_i, \ 1 \leq i \leq k.$ (Generally, we think of $o_i$ as being near $\inner U_i).$ We say that $x$ tends to infinity in $U_i$ if the distance between $x$ and $o_i$ (taken in $U_i$) tends to infinity. 
\end{defin}

\begin{thm}\label{green-ends}
The following dichotomy takes place regarding the behaviors of each of the Green functions $G_{U_i},\ 1\leq i \leq k:$ 
\begin{itemize}
\item [(E1)] If $M_i$ is non-parabolic, then $G_{U_i}(x,y) \to 0$ as $x \to \infty$ in $U_i,$ uniformly for all $y$ in a fixed compact set.
\item [(E2)] If $M_i$ is parabolic, there exists an increasing, unbounded function $f$ taking the positive reals to the positive reals such that for all $R>0$ sufficiently large, there exists a point $x_R$ satisfying $R/2 < d(o_i, x_R) < 3R/2$ and $u_i(x_R) \geq f(R).$

Moreover, fix $\delta >0.$ If $x$ (or, equivalently, $y$) is in a fixed compact set, then $G_{U_i}(x,y)$ is bounded above uniformly, provided $d(x,y) \geq \delta >0$.

\end{itemize}
\end{thm}

\begin{proof}
Fix $i \in \{1, \dots, k\}$ and a point $o_i \in U_i.$ Throughout this proof, we will assume that $d$ refers to the distance in $U_i$ and $B(x,r) =B_{U_i}(x,r), V(x,r) = V_{U_i}(x,r)$ refer to balls and their volumes in $U_i.$ 

\emph{Proof of (E1):} Since $M_i$ is non-parabolic, it possesses a Green function $G_{M_i},$ which satisfies the estimate (\ref{TSGreen}) and condition (\ref{H-par}) found in Appendix \ref{parabolic}. Since the heat kernel is an increasing function of sets, $p_V(t,x,y) \leq p_U(t,x,y)$ if $V \subseteq U$ for all $t>0, \ x,y \in V,$ and, moreover, the Green function is also an increasing function of its domain. Hence
\[ G_{U_i}(x,y) \leq G_{M_i}(x,y) \leq C \int_{d^2(x,y)}^\infty \frac{dt}{V_{M_i}(y,\sqrt{t})} < + \infty, \quad \forall x,y \in U_i, \ x \not = y.\]
Since this last integral converges, it tends to zero as $d(x,y)$ tends to infinity.  As $M_i$ is doubling, $V_{M_i}(y, \sqrt{t})$ and $V_{M_i}(y^*, \sqrt{t})$ are comparable for all $y, y^*$ in a fixed compact set. It follows that $G_{U_i}(x,y)$ tends to zero as $x \to \infty$ uniformly for all $y$ in this fixed compact set. 

\emph{Proof of (E2), first statement:} The situation is more complicated when $M_i$ is parabolic, since in this case $M_i$ possesses no Green function. Consider instead $E_i,$ which possesses a Green function $G_{E_i}$ since $\partial E_i \not = \emptyset.$ Recall $E_i$ differs from $M_i,$ a Harnack manifold, by a compact set $K_i,$ which is the setting of  \cite{GS3}. 

As $U_i$ is uniform in $M_i$ and $\overline{E}_i = \overline{U}_i, \ E_i$ is itself uniform in $M_i$ and hence possesses a harmonic profile $w_i.$ The weighted space $(U_i, \sigma w_i^2)$ remains uniform and hence has a profile $v_i.$ Consider the product $w_i v_i.$ This function must vanish along \emph{all} of $\partial U_i$ since $w_i$ is a harmonic function vanishing on $\partial E_i = \partial^{\text{inner}} U_i$ and $v_i$ must vanish along $\partial^{\text{side}} U_i.$ Moreover, $w_i v_i$ has vanishing normal derivative along $\delta U_i$ as this is true of both $w_i$ and $v_i.$ Since $v_i = (w_i v_i)/w_i$ is the profile of $(U_i, \sigma w_i^2),$ the function $w_i v_i$ must be locally harmonic in $(U_i, \sigma);$ this can be seen by considering the unitary map $T: L^2(U_i, \sigma w_i^2) \to L^2(U_i, \sigma)$ given by $g \mapsto g w_i.$ Therefore $w_i v_i$ is a profile of $(U_i, \sigma).$ Such profiles are unique up to constant multiples, so we may take $u_i = w_i v_i.$ 

Uniformity of $U_i$ in $M_i$ also guarantees, as in \cite[Lemma 3.20]{GyS}, that for any $R>0,$ there exists a point $x_R \in U_i$ such that $d(o_i, x_R)$ and $d(x_R, \partial U_i)$ are both of scale $R.$ In particular, we can take $x_R$ such that 
\[ R/2 < d(o_i,x_R) <3R/2 \text{ and }d(x_R, \partial U_i) \geq c_o R/8\]
for some fixed constant $c_0.$ 
It follows from the proof of Theorem 4.17, \cite{GyS}, that there exists a constant $C>0$ such that 
\[ v_i(y) \leq Cv_i(x_R) \quad \forall \, R>0, \ y \in B(x_R,R).\]
Additionally, from the construction of $u_i$ in \cite{GyS}, there exists a point $y^* \in U_i$ such that $v_i(y^*) =1.$ Thus $C^{-1} \leq v_i(x_R)$ for all $R>0$ sufficiently large. 

Moreover, the proof of Lemma 4.5 in \cite{GS3} implies that for $d(o_i,x) =r$
\[ w_i(x) \approx \int_{r_0}^r \frac{s \, ds}{V(o_i,s)},\]
for any sufficiently large $r_0 > 0,$ where $f \approx g$ means there exist constants $0< c^* \leq C^* < + \infty$ such that $c^* f \leq g \leq C^* f.$ Since $M_i$ is Harnack and parabolic, the above integral tends to infinity as $r$ tends to infinity. Hence $w_i$ tends to infinity in $U_i.$ 

Therefore $u_i(x_R) = v_i(x_R)w_i(x_R) \to \infty$ as $R \to \infty$ since this is true of $w_i$ and $v_i$ is bounded below at the points $x_R.$ Hence we can construct a function $f(R)$ of the type necessary to satisfy (E2). 

\emph{Proof of (E2), second statement:} Since $G_{U_i} \leq G_{E_i},$ it suffices to prove the statement for $G_{E_i}.$ 

By Theorem 5.13 of \cite{GyS}, 

\[ G_{E_i}(x,y) \approx h(x)h(y) \int_{d(x,y)^2}^ \infty \frac{dt}{V_{h}(x, \sqrt{t})} , \]

where $h$ is the profile for $E_i,$ and
\[ V_{h}(x, \sqrt{t}) := \int_{B(x,\sqrt{t})} h^2(z) d\mu(z),\]

where $B(x,\sqrt{t}) = \{ y \in U_i: d_{U_i}(x,y) < \sqrt{t} \}.$ 

Moreover, by Theorem \ref{HarnackEnds} (see also \cite[Theorem 4.17]{GyS}),
\[ V_{h}(x, \sqrt{t} ) \approx [h(x_{\sqrt{t}})]^2\, V(x, \sqrt{t}), \]
where $V$ is the usual volume in $U_i$, and $x_{\sqrt{t}}$ is any point in $U_i$ satisfiying 
\[ d(x, x_{\sqrt{t}}) \leq \frac{\sqrt{t}}{4} \quad  \text{and} \quad d(x_{\sqrt{t}}, M_i \setminus E_i) \geq c_0 \frac{\sqrt{t}}{8}.\]

Therefore 
\[ G_{E_i}(x,y) \approx  h(x)h(y) \int_{d(x,y)^2}^\infty \frac{dt}{[h(x_{\sqrt{t}})]^2 \, V(x, \sqrt{t})}.\]

On the other hand, we recall that \cite{GS3} implies
\[ h(z) \approx \int_{r_0}^{d(o_i,z)} \frac{t \; dt}{V(o_i,t)},\]
where $r_0$ is such that $B_{M_i}(o_i, r_0)$ contains $\partial E_i.$ 

Therefore, by adding a correction term near zero, we can write
\[ h(x_{\sqrt{t}}) \approx \int_0^{d(o_i, x_{\sqrt{t}})} \frac{s e^{-1/s}}{V(s)} \, ds,\]
where $V(s) := V(o_i, s).$ 

For any $x \in E_i,$ set $|x| := d(o_i,x).$ Hence, for $x$ fixed, 
\begin{align*}
G_{E_i}(x,y) &\approx h(x) h(y) \int_{d(x,y)^2}^\infty \frac{dt}{\Big[ \int_0^{|x_{\sqrt{t}}|} \frac{s e^{-1/s}}{V(s)} \, ds \Big]^2 V(x, \sqrt{t}) }\\
& \approx h(x) \Big( \int_0^{|y|} \frac{se^{-1/s}}{V(s)} \, ds \Big) \bigg( \int_{d(x,y)^2}^\infty \frac{dt}{\Big[ \int_0^{|x_{\sqrt{t}}|} \frac{s e^{-1/s}}{V(s)} \, ds \Big]^2 V(x, \sqrt{t}) }\bigg).
\end{align*}

Let $R := d(x,y).$ Since $x$ is fixed and we can always pick $x_{\sqrt{t}}$ so that $d(x,x_{\sqrt{t}})$ is of order $\sqrt{t},$ we have $|x_{\sqrt{t}}| \approx \sqrt{t}.$ Additionally, $|y| \approx d(x,y) \approx R$ for $R$ sufficiently large. Also, $V(x, R) \approx V(R).$  

We define two functions, $f, g : [0, \infty) \to \mathbb{R},$ as follows: 
\begin{align*}
F(R) := \int_{R^2}^\infty  \frac{dt}{\Big[ \int_0^{\sqrt{t}} \frac{s e^{-1/s}}{V(s)} \, ds \Big]^2 V(\sqrt{t}) }; \qquad 
G(R) := \frac{1}{\int_0^R \frac{s e^{-1/s}}{V(s)} \, ds }.
\end{align*}

Then 
\begin{align*}
F'(R) =  \frac{-2R}{\big[\int_0^R \frac{se^{-1/s}}{V(s)} \, ds \big]^2 V(R)}; \qquad
G'(R) =  \frac{-R e^{-1/R}}{\big[ \int_0^R \frac{s e^{-1/s}}{V(s)}\, ds \big]^2 V(R)}
\end{align*}
so that $F'(R) =2 e^{1/R} G'(R).$

Thus, for $R$ large enough, $F'(R), G'(R) < 0.$ Hence both $F$ and $G$ are decreasing, and $F$ is decreasing at least as quickly as $G$, which implies $F(R) \leq G(R)$ for large $R.$ This plus the estimates above yield  
\[ G_{E_i}(x,y) \approx h(x) \frac{F(R)}{G(R)} \approx \frac{F(R)}{G(R)},\]
since $x$ is fixed. Therefore $G_{E_i}$ is largest for small $R = d(x,y)$ and decreases with $R.$ Moreover, for $x$ in some fixed compact set, we can again treat $h(x)$ as constant, which implies $G_{E_i}$ is bounded in the desired fashion. 

\end{proof}

\begin{lem}\label{green-ends-global}
Under the hypotheses of Theorem \ref{thm-profile}, if $G_{U_i}$ satisfies (E1), then (E1) holds for $G$ in $U_i,$ and the same holds for (E2). 
\end{lem}

\begin{proof}
Let $\mathcal{O}_1, \mathcal{O}_2$ be two precompact sets with smooth boundary in $M$ such that $K \subset \mathcal{O}_1 \subset \mathcal{O}_2$ and $\overline{\mathcal{O}}_1 \subset \mathcal{O}_2.$ Let $K_i = \overline{\mathcal{O}}_i, \ i = 1,2.$ We will show there exist constants $0 < c \leq C < + \infty$ such that, for any $y \in U_i \cap K_1,$ 
\[ cG_{U_i}(\cdot, y) \leq G(\cdot,y) \leq C G_{U_i}(\cdot,y)\]
on $U_i \setminus K_2.$ 

This implies $G_{U_i}, G$ are comparable in the end $U_i$ sufficiently far away from the middle. Since $K_2$ is compact, the behavior of $G$ near the middle is determined by the boundary of $\Omega$ and the fact that $G \approx G_{U_i}$, as local elliptic and boundary Harnack inequalities hold. (See Appendix \ref{harmonic} for statements of Harnack inequalities.) Thus it suffices to prove $G_{U_i} \approx G$ as above.

As $U_i \subset \Omega, \ G_{U_i}(x,y) \leq G(x,y)$ for all $x,y \in U_i,$ proving the first inequality with $c=1.$ The other inequality is more challenging. 

Let $\{\Omega_j\}_{j=1}^\infty$ be an exhaustion of $\Omega$ by pre-compact open sets and $G_{\Omega_j}$ be the Green function for $\Omega_j, \ j=1,2,3,\dots.$ As $G_{\Omega_j} \nearrow G,$ it suffices to show there exists $0<C<+\infty$ such that $G_{\Omega_j}(\cdot,y) \leq C G_{U_i}(\cdot,y)$ in the desired range, where $C$ does not depend on $j.$

Fix $\varepsilon > 0.$ Let $o_i$ be a fixed reference point in $U_i^\bullet \cap K_1.$ Assume $o_i \in \Omega_j$ and $K_2 \subset \Omega_j$ for all $j.$ As in the previous theorem, let $d, B,$ and $V$ refer to distance, balls, and volumes taken in $U_i.$ Locally in a coordinate chart neighborhood of $o_i,$ the functions $G_{\Omega_j}$ behave like the Green function of $\mathbb{R}^n,$ where $n$ is the dimension of $M.$ For some $r_0 >0,$ we may take $W = B(o_i, r_0) \subset K_1$ to be our coordinate chart neighborhood. Then for all $z$ such that $d(o_i, z) = r_0/2$ and $d(z, \partial \Omega_j) \geq \varepsilon,$ there exist constants $0< c_0\leq C_0<+\infty$ independent of $j$ such that 
\begin{equation}\label{G_constant} c_0  \leq G_{\Omega_j}(o_i, z) \leq C_0.\end{equation}

By construction of $K_1, K_2,$ there exists $\delta >0$ such that if  $x \in U_i \setminus K_2, \linebreak y \in K_1,$ then $ d(x,y) \geq \delta.$  Since $G_{\Omega_j}$ is harmonic, $K_2$ is compact, and (\ref{G_constant}) holds, by the local elliptic Harnack inequality in $\Omega,$ there exist constants $0< c_1 \leq C_1 < +\infty$ which do not depend on $j$ such that 
\[ c_1 \leq G_{\Omega_j}(x,y) \leq C_1 \]
for all $x \in \partial K_2 \cap U_i, \ y \in K_1, \ d(x, \partial \Omega_j) \geq \varepsilon, \ d(y, \partial \Omega_j) \geq \varepsilon.$

Using the boundary Harnack inequality to compare $G_{\Omega_j}$ to $G_{U_i}$ along points of $\partial K_2 \cap U_i$ at distance less than $\varepsilon$ from $\partial \Omega_j$ and to push $y \in K_1 \cap U_i$ toward $\partial \Omega,$ and using the elliptic Harnack inequality to gain control of $G_{U_i}$ away from the Dirichlet boundary, we see there exists a constant $0<C<+\infty$ such that 
\[ G_{\Omega_j}(\cdot, y) \leq C G_{U_i}(\cdot, y) \]
on $\partial K_2 \cap U_i$ for any $y \in K_1 \cap U_i.$ 

We now use a comparison principle, since both $G_{\Omega_j}, G_{U_i}$ are harmonic in $\Omega_j \cap (U_i \setminus K_2).$ Along $\partial K_2 \cap U_i,$ we showed $G_{\Omega_j}(\cdot, y) \leq C G_{U_i}(\cdot, y).$ Also, $G_{\Omega_j}$ vanishes along the inner boundary of $\Omega_j$ that lies in $U_i,$ while $G_{U_i}$ is positive there. Both functions vanish along $\partial \Omega_j \cap \partial U_i$ and have vanishing normal derivative along $\delta \Omega_j \cap U_i.$ The Hopf boundary lemma guarantees that minimums of harmonic functions cannot occur solely at points where the normal derivative vanishes, and therefore 
\[ G_{\Omega_j} (\cdot, y) \leq C G_{U_i} (\cdot, y)\]
on $\Omega_j \cap (U_i \setminus K_2),$ where $y \in K_1 \cap U_i.$ Taking $j \to \infty$ finishes the proof. 

\end{proof}

\subsection{Construction of the Profile for \texorpdfstring{$\Omega$}{Omega}} We now prove the main theorem, Theorem \ref{thm-profile}, in this section. The construction of the profile $h$ of $\Omega$ closely follows the method of Sung, Tam, and Wang in \cite{STW}. 

\begin{proof}[Proof of Theorem \ref{thm-profile}.]
For clarity, the proof is divided into a series of steps. 

\emph{Step 1 (Construct a global harmonic function on $\Omega$):} Fix a point $o \in K$ and take precompact open sets $\mathcal{O}_1,\ \mathcal{O}_2 \subset M$ with smooth boundary such that  $K\subset \mathcal{O}_1 \subset \mathcal{O}_2$ and no points in $\overline{\delta \Omega} \cap \partial \Omega$ belong to the set  $\mathcal{O}_2 \setminus \mathcal{O}_1,$ which is possible since every point in $M$ possesses a neighborhood containing only finitely many components of $\partial \Omega.$

Construct a smooth function $\psi$ on $\Omega$ such that $\psi \equiv 1$ on $\Omega \setminus \mathcal{O}_2, \ \psi \equiv 0$ on $\mathcal{O}_1,$ and $\psi$ has vanishing normal derivative on $\delta \Omega.$ Let $u$ be defined on $\Omega \setminus K$ by $u|_{U_i} = u_i, 1 \leq i \leq k.$ 

Define 
\[ h(x) = (u\psi)(x) + \int_\Omega G(x,y) \Delta (u \psi)(y) \, d\mu(y) , \quad \forall x \in \Omega.\]

Here $\Delta(u \psi)$ is a smooth function with with compact support on $M$ by the construction of $\psi$ and elliptic regularity theory (it extends smoothly to the boundary by construction of $\mathcal{O}_1,\ \mathcal{O}_2).$ The relevant weak definitions regarding harmonic functions may be found in Appendix \ref{harmonic}; here, for simplicity, we write the proof in terms of the corresponding classical definitions.

For any smooth, compactly supported function $\alpha$ on $M,$ set 
\[ G(\alpha) := \int_\Omega G(x,y) \alpha(y) \, d\mu(y).\]
Then for $\alpha = \Delta(u\psi), \ h = u\psi+ G(\alpha).$ 

We compute 
\[ \Delta G(\alpha) = \int_\Omega \Delta G(x,y) \alpha(y) \, d\mu(y) = \int_\Omega \Delta_\Omega G(x,y) \alpha(y) \, d\mu(y) = -\alpha(x),  \]
as we may replace $\Delta$ (which applies to smooth functions) by $\Delta_\Omega$ (the infinitesimal generator associated with $(W_0(\Omega), \int g(\nabla f, \nabla f) \, d\mu)$) since $G$ is the Green function for $\Omega.$

As $\delta \Omega$ is smooth, a direct calculation shows the normal derivative of $h$ on $\delta \Omega$ vanishes, since this is true of all of $u, \psi,$ and $G$ by definition. Similarly, as $u = 0$ on $\partial \Omega \setminus K$ and $G = 0$ on all of $\partial \Omega,$ it follows that $h =0$ on $\partial \Omega$ as well. Thus $h$ is a global harmonic function on $\Omega.$

\emph{Step 2 (On each end $U_i, \ h$ behaves similarly to $u_i$):} For the remainder of the proof, we need to make use of the two possible cases of the behavior of $G$ on each $U_i, \ 1\leq i \leq k,$ given by Theorem \ref{green-ends} and Lemma \ref{green-ends-global}. Crucially, these two conditions imply that $h$ behaves similarly to $u_i$ in (at least) one of two (non-equivalent) ways. 

First suppose $U_i$ satisfies (E1). Note $\alpha = \Delta(u \psi)$ is bounded and the integral in $G(\alpha)$ is only over the compact set $\overline{\mathcal{O}}_2.$ Additionally, (E1) implies $G$ tends to zero with $x$. These three facts imply $G(\alpha) \to 0$ at infinity in the end $U_i.$ Since $\psi \to 1$ at infinity, $h - u_i \to 0$ at infinity in $U_i.$

If $U_i$ satisfies (E2), then we claim $h/u_i \to 1$ at infinity in the end $U_i.$ To see this, take $R>0$ sufficiently large and consider the annulus $A_R = \{ x \in U_i : R/2 < d(o, x) < 3R/2\}.$ The key step in proving this claim is to show that $u_i$ is actually not too small in the entire annulus $A_R$ by obtaining a lower bound for $u_i$ depending on $G$ and $R.$

We will make use of the technique of remote and anchored balls found in \cite{GS4}, to which we refer the reader for more details. In brief, an anchored ball is simply a ball whose center belongs to $\partial \Omega$, while a remote ball (in $\Omega$) is one whose double is precompact in $\Omega$. The import of this is that elliptic Harnack inequalities hold in remote balls, whereas boundary elliptic Harnack inequalities hold in anchored balls. 

 By assumption (E2), there exist points $x_R \in A_R$ and an increasing, un-bounded real function $f$ such that $u_i(x_R) \geq f(R)$ for all $R>0$ sufficiently large. Since $U_i$ satisfies (RCA) by (C1), for any $x^* \in A_R$ and any fixed $\varepsilon >0,$ there exists a sequence of at most $Q_\varepsilon$ balls connecting $x^*$ and $x_R$ where each ball is either remote of radius $\frac{\varepsilon R}{4}$ or anchored to the boundary of radius $\varepsilon R.$ Label this sequence of balls $B_0, \dots, B_l$ in such a way that $x_R \in B_0, x^* \in B_l$ and $B_j \cap B_{j+1} \not = \emptyset, \ 0\leq j \leq l-1.$ 
 
 Recall that (C2) indicates an elliptic boundary Harnack inequality holds uniformly in $U_i,$ while hypothesis (H1) implies an elliptic Harnack inequality holds uniformly in $M_i.$ Let $C_H$ be the uniform elliptic Harnack inequality constant for $M_i$ and $C_B$ denote the uniform elliptic boundary Harnack constant for $U_i.$ We show that $u_i$ remains relatively large in all of $B_0.$ 
 
 If $B_0$ is a remote ball and $y \in \mathcal{O}_2,$ applying the elliptic Harnack inequality to the non-negative harmonic functions $G$ and $u_i$ in $B_0,$ 
 \[ G(x,y) \leq C_H^2 \frac{G(x_R,y)}{u_i(x_R)} u_i(x) \leq \frac{C_H^2 L}{f(R)} u_i(x) \quad \forall \  x \in B_0,\]
 where $L$ is an upper bound on $G(x_R, y)$ for large $R$ given by Theorem \ref{green-ends}. 
 
 Similarly, if $B_0$ is anchored to the boundary, we may compare $G$ and $u_i$ using the elliptic boundary Harnack inequality. Although this inequality a priori applies only in the ball of half the radius, by covering $B_0$ by a finite number of remote or anchored balls and chaining appropriate Harnack inequalities, the boundary Harnack inequality actually holds in all of $B_0,$ albiet with a potentially different constant, which we continue to call $C_B.$ Thus
  \[ \frac{G(x,y)}{u_i(x)} \leq C_B \frac{G(x_R,y)}{u_i(x_R)} \leq \frac{C_B L}{f(R)} \quad \forall \ x \in B_0.\]
  
 In either case, we obtain a lower bound for $u_i,$ involving $G,$ in the entire ball $B_0.$ We then chain between the balls $B_0, \dots, B_l,$ obtaining a similar inequality with an additional constant at each stage. Since there are at most $Q_\varepsilon$ balls, there exists a constant $0< C < +\infty$ depending only on $Q_\varepsilon, C_H, C_B,$ and $L$ such that 
 \[ G(x^*, y) \leq \frac{C}{f(R)} u_i(x^*) \quad \forall \ x^* \in A_R.\]
 
Recall $h= u\psi + G(\alpha),$ where $\alpha$ is bounded. Hence for $d(o,x) = |x| \approx R$ large enough in $U_i,$
\[ u_i(x) - \delta_R u_i(x) \leq h(x) \leq u_i(x) + \delta_R u_i(x)\]
for some $\delta_R >0$ that tends to zero as $R$ tends to infinity. Thus $h/u_i \to 1$ as $x \to \infty$ in $U_i$ as claimed.

\emph{Step 3 ($h$ is non-negative):} Let $\varepsilon >0.$ Then there exists a compact set $K_\varepsilon \subset M$ such that on $\Omega \setminus K_\varepsilon,$ $-\varepsilon \leq u_i - \varepsilon \leq h$ on ends $U_i$ where $U_i$ satisfies (E1) and $0 < (1/2) u_i \leq h$ at points in ends $U_i$ satisfying (E2), and every end falls in (at least one) of these two cases. Recall $h=0$ on all of $\partial \Omega,$ and by the Hopf boundary lemma, a minimum of $h$ cannot occur only on $\delta \Omega.$ Hence a minimum principle implies $-\varepsilon<h$ on all of $\Omega.$ Since $\varepsilon$ was arbitary, we conclude $h \geq 0$ on $\Omega.$

\emph{Step 4 ($h \geq cu_i$ on $U_i$):} We again employ a minimum principle. For fixed $i,$ first assume that $U_i$ satisfies (E1). For every $\varepsilon >0,$ since $h-u_i \to 0$ at infinity in $U_i,$ there exists a compact set $K_\varepsilon$ such that $u_i - \varepsilon \leq h$ in $U_i \setminus K_\varepsilon.$ Since $h$ is non-negative and $u_i$ vanishes along $\inner U_i$ by definition, $u_i \leq h$ there. Both $u_i$ and $h$ vanish along $\side U_i,$ and both have vanishing normal derivative along $\delta U_i.$ Take a sequence of balls $B(o, R_l)$ such that $R_l \to \infty$ as $l \to \infty$ and $K_{1/l} \setminus \partial \Omega$ is contained in $B(o, R_l)$ for $l=1,2,3, \dots.$ Then on $U_i \cap \partial B(o, R_l), \ u_i - 1/l \leq h.$ Thus on $U_i \cap B(o, R_l)$ the weak minimum principle combined with the Hopf boundary lemma gives $u_i - 1/l \leq h.$ Sending $l \to \infty$ yields $u_i \leq h$ on $U_i.$ 

Assume $U_i$ satisfies (E2) instead. Then we may choose $R_0>0$ such that $(1/2) u_i \leq h$ in $U_i \setminus \overline{B(o,R_0)}.$ Then for $R$ sufficiently large, $h=u_i=0$ on $\side U_i \cap B(o,R)$ and, along $\inner U_i$ and $\partial B(o,R) \cap U_i, \ (1/2)u_i \leq h.$ Therefore the Hopf boundary lemma and a minimum principle give $(1/2) u_i \leq h$ on $B(o,R) \cap U_i.$ Consequently $(1/2) u_i \leq h$ on all of $U_i.$ 

Combining the two cases above, $(1/2) u_i \leq h$ on $U_i, \ 1\leq i \leq k,$ so in each end, $h$ grows at least as fast as the harmonic profile for that end. 

\emph{Step 5 ($h$ is positive on $\Omega$):} As a local elliptic Harnack inequality holds in $\Omega,$ either $h \equiv 0$ on $\Omega$ or $h >0$ on $\Omega.$ Since $u_i > 0$ in $U_i,$ the previous step implies $h \not \equiv 0.$ Thus $h$ is positive and hence is a profile for $\Omega.$

\end{proof}

\subsection{Relationship between \texorpdfstring{$h$}{h} and the \texorpdfstring{$u_i$}{profiles of the ends}} By virtue of Theorem \ref{thm-profile}, $\Omega$ possesses a profile $h$ which must grow at least as fast as the profile $u_i$ in $U_i, \ 1 \leq i \leq k.$ In fact, outside of a compact set in $U_i,$ the profiles $h$ and $u_i$ are comparable. This comparison is crucial for our main result; the existence of a non-negative harmonic function on $\Omega$ satisfying the appropriate boundary conditions (but no other properties) follows from \cite{JLLSC2} due to the smoothness of the boundaries $\partial \Omega, \delta \Omega.$

\begin{thm}\label{h-u_i-comp}
Assume $\partial \Omega \not = \emptyset$ and let $h$ be a harmonic profile for $\Omega$ as constructed in Theorem \ref{thm-profile}. Then there exist constants $0< c_i \leq  C_i < \infty$ and compact sets $\hat{K}_i \subset M, \ 1 \leq i \leq k$ such that 
\[ c_i u_i \leq h \leq C_i u_i \]
on $U_i \setminus \hat{K}_i.$
\end{thm}

\begin{proof}
For simplicity, we drop the subscript $i.$ The desired lower bound holds with $c =1/2$ on all of $U$ as in Theorem \ref{thm-profile}. Once again the proof depends on the relative behavior of the Green function.

Assume $U$ satisfies (E1). Let $\hat{K}$ be a compact subset of $M$ such that the inner boundary of $U$ is contained in the interior of $\hat{K}.$ Recall an elliptic Harnack inequality holds locally on $\Omega,$ and since $U$ is uniform in a Harnack manifold, by consequence (C2) an elliptic boundary Harnack inequality holds uniformly in $U.$ For $\hat{U} := U \setminus \hat{K},$ consider $\inner \hat{U}.$ This is a compact set and hence it can be covered by a finite number of balls that are either far away from $\side \hat{U}$ or that are near this boundary. In balls far away from $\side \hat{U},$ the elliptic Harnack inequality implies both $h$ and $u$ are relatively constant. As we approach $\side \hat{U},$ since an elliptic boundary Harnack inequality holds uniformly, $h$ and $u$ decay at the same rate. Hence there exists $C \geq 1$ such that $h \leq C u$ along $\inner \hat{U}.$ 

Since $h-u \to 0$ at infinity in $U,$ there exists a sequence of balls $B(o, R_l)$ such that $h \leq u + 1/l$ on $\hat{U} \backslash B(o, R_l).$ It follows from a minimum principle and the Hopf boundary lemma that $h\leq Cu + 1/l$ on $B(o, R_l) \cap \hat{U}.$ Sending $l\to \infty$ gives $h \leq Cu$ on $\hat{U}.$

If instead $U$ satisfies (E2), the desired statement follows immediately from the fact that $h/u \to 1$ at infinity in $U.$

\end{proof}

\section{Heat Kernel Estimates}\label{prove_thm}

This section explains how the profile $h$ of $\Omega$ (from Theorem \ref{thm-profile}) can be used to estimate the mixed boundary condition heat kernel on $\Omega$. The key techniques used here are those of \cite{GS1,GS5} (dealing with manifolds with ends) and of \cite{GyS} (dealing with mixed Dirichlet and Neumann boundary conditions). Appendix \ref{earlier_results} explains additional connections with the existing literature.

 \subsection{The \texorpdfstring{$h$}{h}-transform space} Assume $\partial \Omega \not = \emptyset$ and let $h$ be a harmonic profile for $\Omega$ as constructed in Theorem \ref{thm-profile}. Consider the weighted manifold $(\Omega, h^2 \sigma).$ Notice this change of measure is related to the unitary map $T: L^2(\Omega, h^2 \sigma) \to L^2(\Omega, \sigma)$ defined by $T(f) = hf$ for all $f \in L^2(\Omega, h^2 \sigma).$ The heat kernel $p_{\Omega, h^2} (t,x,y) = p_h (t,x,y)$ for $\Omega$ after $h$-transform (that is, $(\Omega, h^2 \sigma)$) is related to the heat kernel $p(t,x,y)$ for $(\Omega, \sigma)$ by the following simple formula \cite{GS3, GyS}:
\[ p(t,x,y) = h(x) h(y) p_h(t,x,y).\]
Hence in order to estimate $p(t,x,y),$ it suffices to estimate $p_h(t,x,y).$ To estimate this quantity, we will use Theorem \ref{HK-ends-estimate}. The results in this section rely heavily upon Appendix \ref{Harnack}.

Let $K^* \subset M$ be compact such that $K$ is a subset of the (topological) interior of $K^*.$ Then $\Omega = K^* \sqcup U^*_1 \sqcup \cdots U^*_k,$ where $U^*_i$ and $U_i$ differ by a compact set for $1\leq i \leq k.$ 

Consider the manifolds $(U^*_i, \sigma h^2),$ \emph{where we put Neumann boundary condition on $\inner U_i^*$} (this amounts to the abuse of notation ``$U_i^* = U_i^* \cup \inner U_i^*$"). We first show these manifolds are in fact Harnack and non-parabolic.

\begin{prop}\label{h-trans-Harnack-ends}
The manifolds $(U_i^*, \sigma h^2), \ 1 \leq i \leq k,$ described in the preceding paragraphs are Harnack. 
\end{prop}

\begin{proof}
For appropriate choice of $K^*,$ Theorem \ref{h-u_i-comp} guarantees $c u_i \leq h \leq Cu_i$ on $U_i^*.$ By consequence (C4) of our hypotheses, the manifold $(U_i, \sigma u_i^2)$ is a Harnack manifold. We can also choose $K^*$ such that the boundaries of the manifolds $U_i^*, \ 1\leq i \leq k,$ satisfy condition (*). Moreover, $h$ vanishes on $\partial U_i^* = \partial \Omega \cap (\partial U_i \setminus K^*)$ by construction. The function $h$ is harmonic in the interior of $U_i^*$ but does not have vanishing normal derivative along $\inner U_i^*.$ However, since such points are part of $U_i^*,$ they do not affect whether functions in $\mbox{Lip}_c(\overline{U}_i^*)$ belong to $W_0^1(U_i^*, h^2 \sigma).$ Thus it follows from Proposition 5.8 of \cite{GyS} that  $\mbox{Lip}_c(\overline{U}_i^*) \subset W_0^1(U_i^*, h^2 \sigma).$ Moreover, $(U_i^*, \sigma h^2)$ must be uniformly doubling and satisfy the Poincar\'{e} inequality uniformly since this is true of $(U_i, u_i^2 \sigma)$ and $h \approx u_i$ in $U_i^*.$ Therefore by Theorem \ref{HarnackNotComplete}, $(U_i^*, h^2 \sigma)$ is a Harnack manifold. 
\end{proof}

\begin{prop}\label{h-trans-nonparabolic}
The manifolds $(U_i^*, \sigma h^2), \ 1 \leq i \leq k,$ are non-parabolic.
\end{prop}

\begin{proof}
One of the equivalent definitions of non-parabolicity is that the space $(U_i^*, \sigma h^2)$ possesses a non-constant, positive superharmonic function \cite{GSurv}. Such a function must satisfy Definition B.1 where we consider $U_i^*$ as an open subset of $\overline{U}_i^*,$ but with the equality replaced by a less than or equal to. In other words, we need a smooth function $u$ such that $\Delta_{U_i^*,\, h^2\sigma} \, u \leq 0$ in the (geometric) interior of $U_i^*$ and $u$ has vanishing normal derivative along the boundary points of $U_i^*$ as a manifold. 

Consider the function $1/h$ on $(U_i^*, \sigma h^2).$ By the correspondence between  $L^2(U_i^*, \sigma)$ and $L^2(U_i^*, \sigma h^2),$ $1/h$ is harmonic in the geometric interior of $U_i^*.$ It also has vanishing normal derivative on $\delta U_i^* \cap \delta U_i,$ since this is true of $h.$ However, $h$ does not have vanishing normal derivative on the inner boundary of $U_i^*,$ so $1/h$ is not itself superharmonic. Nonetheless, $1/h$ is a local harmonic function in $U_i^*$ in the desired sense outside of a compact set containing $\inner U_i^*$, so $1/h$ can be extended to a positive superharmonic function in $U_i^*.$ Since $h$ behaves like $u_i,$ if this modified version of $1/h$ is constant, then $u_i$ must be relatively constant on $U_i^*$. However, if this is the case, then $\overline{U_i}$ must have been non-parabolic to start with, so $U_i^*$ must also be non-parabolic.

\end{proof}

We now come to the main result of this paper. In Theorem \ref{main_thm} below, all notation will be as in Theorem \ref{HK-ends-estimate} where subscripts $h$ indicate we have applied the theorem to the manifold $(\Omega, h^2 \sigma).$ For explicit examples of the estimates in Theorem \ref{main_thm}, see Section \ref{exs}. 

 Before the main result, we recall some notation (see also Theorem \ref{HK-ends-estimate}). For $x \in \Omega,$ the notation $i_x$ indicates to select the index $i$ such that $x$ belongs to the end $U_i.$ If $x \in K,$ we set $i_x = 0$. Also let $|x| := \sup_{y \in K} d(x,y).$ The distance $d_+$ indicates distance passing through the compact middle $K,$ whereas $d_\emptyset$ indicates distance avoiding $K.$ We will also need the following two functions involving the $h$-transform.

\begin{defin}
If $B_i(x,r)$ denotes a ball in $U_i$ centered at $x$ with radius $r$ and $o_i$ is a fixed reference point on $\inner U_i, \ 1 \leq i \leq k,$ then
\begin{equation}\label{V_i_h}V_{i,h}(r) = V_{i,h}(o_i,r) = \mu_{i,h}(B_i(o_i, r)) := \int_{B_i(o_i,r)} h^2(x) \sigma_i(x) dx\end{equation}
and $V_{0,h}(r) = \min_{1 \leq i \leq k} V_{i,h}(r).$

Furthermore, we set
\begin{equation}\label{H_h}
 H_h(x,t) = \min \Bigg\{ 1, \frac{|x|^2}{V_{i_x,h}(|x|)} + \Big(\int_{|x|^2}^t \frac{ds}{V_{i_x,h}(\sqrt{s})}\Big)_+ \Bigg\}.
 \end{equation} 
\end{defin}

\begin{thm}\label{main_thm}
Let $(M, \sigma)$ be a weighted complete Riemannian manifold with boundary such that $M= M_1 \# \cdots \# M_k.$ Let $\Omega \subset M$ be open such that $\partial \Omega = M \setminus \Omega \subset \delta M$ satisfies condition (*) and assume $\partial \Omega \not = \emptyset.$ 

Let $K$ be a compact set such that $M = K \sqcup (E_1 \sqcup \cdots \sqcup E_k)$ and $U_i = \Omega \cap E_i, \ 1 \leq i \leq k.$ Assume (H1) and (H2) so that $(M_i, \sigma_i),\ 1 \leq i \leq k$ are Harnack manifolds and each $U_i$ is uniform in $M_i.$ Let $h$ denote a profile for $\Omega$ as constructed in Theorem \ref{thm-profile}. 

Then for any $t > 1, \ x,y \in \Omega,$ we have the estimate 
\begin{align*}
p&(t,x,y) \approx \ Ch(x) h(y) \Bigg[\frac{1}{\sqrt{V_{i_x,h}(x,\sqrt{t})V_{i_y,h}(y, \sqrt{t})}} \exp \Big(-c\frac{d^2_\emptyset(x,y)}{t}\Big) \\
&+\Bigg(\frac{H_h(x,t) H_h(y,t)}{V_{0,h}(\sqrt{t})} + \frac{H_h(y,t)}{V_{i_x,h}(\sqrt{t})} + \frac{H_h(x,t)}{V_{i_y,h}(\sqrt{t})} \Bigg) \, \exp \Big( -c \frac{d_+^2(x,y)}{t}\Big)\Bigg],
\end{align*}
where the constants $C,c$ are different in the upper and lower bounds.
\end{thm}

\begin{proof}
By Propositions \ref{h-trans-Harnack-ends} and \ref{h-trans-nonparabolic}, the ends $U_i^*$ are non-parabolic and Harnack in the sense of Theorem \ref{HarnackNotComplete}. Moreover, the restriction of any Lipschitz function with compact support in $M$ will lie in $W_0^1(\Omega, \sigma h^2)$ by Proposition 5.8 of \cite{GyS} since $h$ is a positive harmonic function in $\Omega,$ vanishing on $\partial \Omega.$

Therefore applying Theorem \ref{HK-ends-estimate} to $(\Omega, \sigma h^2),$ with ends $U_i^*, \ 1 \leq i \leq k,$ gives an estimate for  $p_h(t,x,y).$ The theorem follows once we recall $p(t,x,y) = h(x) h(y) p_h(t,x,y).$
\end{proof}

\begin{rem}We now indicate how to compute some of the quantities in Theorem \ref{main_thm} in practice. In fact, all such quantities can be computed based solely on information about the ends $x$ and $y$ belong to \emph{except} for the quantity $V_{0,h}(r).$ 

By Theorem \ref{h-u_i-comp}, in each end $U_i,$ the profile $h$ of $\Omega$ is comparable to the profile $u_i$ of that end, away from some compact set $K^* \supset K$. Hence we can compute $h$ using these profiles. As $h$ is harmonic, inside of $K^*$ it is roughly constant away from points of $\partial \Omega$ and vanishes linearly as it approaches such points. Frequently, given an end $U_i,$ it may be easier to compute the profile of some set $V_i$ that is close to $U_i$ in the sense that their difference is a compact subset of $K^*$. Using Harnack inequalities and maximum principles as in Section \ref{construct}, we see the profiles of such $U_i, V_i$ are comparable.

A useful technique for computing quantities $V_{i_x,h}(x, \sqrt{t})$ in the theorem above is the use of points $x_{\sqrt{t}},$ which were encountered in the proof of Theorem \ref{green-ends}. The spirit is the same as that of Theorem \ref{HarnackEnds}, which does not directly apply. For any $t>0$ and any point $x \in U_i, \ 1 \leq i \leq k,$ there exists a point $x_{\sqrt{t}} \in U_i$ such that $d(x, x_{\sqrt{t}}) \leq \sqrt{t}/4$ and $d(x_{\sqrt{t}}, M_i \setminus E_i) \geq c_0 \sqrt{t}/8$ for some constant $c_0 >0$ \cite[Lemma 3.20]{GyS}. As the $U_i$'s are uniform, by Theorem 4.17 of \cite{GyS} (or Theorem \ref{HarnackEnds}), we have
\[\int_{B_i(x, \sqrt{t})} u_i^2(y) \sigma_i(y)dy \approx u_i(x_{\sqrt{t}})^2 \, V_i(x, \sqrt{t}) \quad \forall x \in U_i.\]
As $u_i \approx h$ in $U_i \setminus K^*$ for $K^*$ compact, it follows that 
\[ V_{i_x, h}(x, \sqrt{t}) \approx h(x_{\sqrt{t}})^2 V_i(x, \sqrt{t}) \quad \forall x \in U_i.\]

In the simplest examples, the integral in the definition (\ref{H_h}) of $H_h(x,t)$ does not contribute and the computation reduces to 
\[ H_h(x,t) \approx \min \Big\{ 1, \frac{|x|^2}{V_{i_x,h}(|x|)}\Big\} \approx \min \Big\{ 1, \frac{|x|^2}{h^2(x_{|x|})V_{i_x}(|x|)} \Big\}.\]
If $V_{i_x, h}$ grows fast enough, then the second term above is always less than $1$ and the computation simplifies further. See Remark \ref{simplify_H} and (\ref{vol_cond}) for the appropriate condition on the volume.
\end{rem}

Obtaining heat kernel estimates for small times $t\leq1$ is much simpler and follows from the fact that the parabolic Harnack inequality (Definition \ref{PHI}) holds for small scales in $(\Omega, \sigma h^2).$
 
\begin{thm}\label{small-times}
Under the hypotheses of Theorem \ref{main_thm}, for any $0<t\leq1$ and $x,y \in \Omega,$
\[ p(t,x,y) \approx h(x) h(y) \frac{1}{V_h(x, \sqrt{t})} \exp \Big(-c \frac{d(x,y)^2}{t}\Big),\]
where $V_h$ denotes the volume in $(M, \sigma h^2)$ and $d$ denotes distance in $M.$
\end{thm}

\begin{proof}
Since $(\Omega, \sigma h^2)$ is a connected sum of the Harnack manifolds $(\Omega_i, \sigma h^2),$ the parabolic Harnack inequality holds up to scale $r_0$ for any $r_0>0$ as in \cite[Lemma 5.9]{GS5}. Thus for any $0<t<1, \ x,y \in \Omega,$
\[ p_h(t,x,y) \approx \frac{1}{V_h(x, \sqrt{t})} \exp \Big(-c \frac{d(x,y)^2}{t}\Big) \]
and the result follows from the relation between $p(t,x,y)$ and $p_h(t,x,y).$
\end{proof}

\begin{rem}
In fact, the estimate in Theorem \ref{small-times} can be replaced by that in Theorem \ref{main_thm} as is explained in \cite{GS5}.
\end{rem}

\section{Examples}\label{exs}

\begin{example}
Suppose $M$ is a connected sum of three cones in $\mathbb{R}^2$ with apertures $\alpha_1, \alpha_2, \alpha_3 \in [0, 2\pi)$ such that $\alpha_1 + \alpha_2 + \alpha_3 < 2\pi.$ (While we should round the corners to stay in the category of smooth manifolds, this changes nothing significant.) For simplicity, we assume that the vertex of each cone of positive aperture is the origin. We consider $\Omega \subset M$ that encodes boundary conditions on these cones; for each cone of positive aperture, we assign one of the following three boundary conditions: either both sides of the cone carry Neumann boundary condition, both sides carry Dirichlet boundary condition, or one side carries each boundary condition. A cone of zero aperture is represented by a strip with Neumann condition on both sides. A typical example of this situation in found in Figures \ref{fig:3cones} and \ref{fig:3cones2}.

\begin{figure}[h]\begin{center}\begin{tikzpicture}[scale=1] 
\draw[blue, ultra thick] (.7071, .7071)--(2,2);
\draw[blue, ultra thick] (-.7071, .7071)--(-2,2);
\draw[blue, dashed, ultra thick] (.7071, .7071)--(0,0);
\draw[blue,dashed, ultra thick] (-.7071, .7071)--(0,0);

\draw[red,ultra thick] (-.966, .259)--(-3.5, .938);
\draw[blue,ultra thick] (-.866, -.5)--(-3, -1.723);
\draw[red, dashed, ultra thick] (-.966, .259)--(0, 0);
\draw[blue, dashed, ultra thick] (-.866, -.5)--(0, 0);

\draw[red, ultra thick] (.866,.5)--(3.5,.5);
\draw[red, ultra thick] (.866,-.5)--(3.5,-.5);

\draw[dashed,ultra thick] (0,0) circle (1);
\node at (0,0)[anchor=north] {$(0,0)$};
\draw[fill=black] (0,0) circle (.05);

\node at (0,1)[anchor=south] {$\alpha_1$};
\draw[ultra thick] (.7071,.7071) arc (45:135:1); 
\draw[ultra thick] (-.966,.259) arc (165:210:1);
\node at (-1,-.2)[anchor=east] {$\alpha_2$};
\node at (1.5,0)[anchor=west] {$\alpha_3 = 0$};
\end{tikzpicture}
\end{center}\caption{An example of a connected sum of three cones whose vertices lie at the origin and which are placed around the unit circle.}\label{fig:3cones}\end{figure}
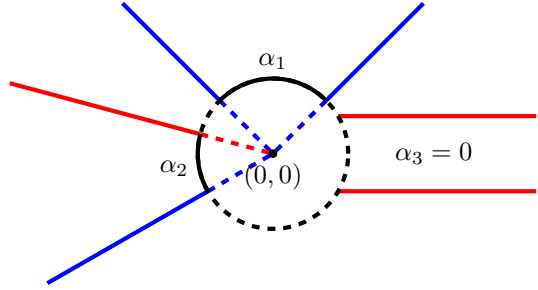

\begin{figure}[h]\begin{center}
\begin{tikzpicture}[scale=.8] 
\draw[blue, ultra thick] (.7071, .7071)--(2,2);
\draw[blue, ultra thick] (-.7071, .7071)--(-2,2);
\draw[red,ultra thick] (-.966, .259)--(-3.5, .938);
\draw[blue,ultra thick] (-.866, -.5)--(-3, -1.723);
\draw[red, ultra thick] (.866,.5)--(3.5,.5);
\draw[red, ultra thick] (.866,-.5)--(3.5,-.5);

\draw [ultra thick, red] (.866,.5) arc (30:45:1);
\draw [ultra thick, blue] (-.7071,.7071) arc (135:165:1);
\draw [ultra thick, blue] (-.866,-.5) arc (210:330:1);

\path[fill, red, opacity =.1] (.7071, .7071)--(2,2)--(-2,2)--(-.7071, .7071) arc (135:165:1) --(-.966,.259)--(-3.5, .938)--(-3, -1.723)--(-.866,-.5) arc (210:330:1)--(.866, -.5)--(3.5,-.5)--(3.5,.5)--(.866,.5) arc (30:45:1);

\node at (-2.7,.1) {$\Omega$};

\end{tikzpicture}\end{center}\caption{The manifold $\Omega$ associated with Figure \ref{fig:3cones}.}\label{fig:3cones2}\end{figure}
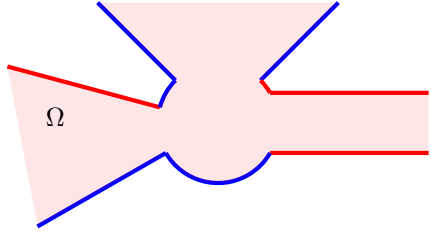

The above six pieces of information (the three apertures of the cones and what boundary conditions they carry on their sides) are all that is necessary to determine the behavior of $p(t,o,o)$ in such domains (where naturally we take $o$ to be the origin). 

We will assume at least one cone has some Dirichlet boundary condition to ensure that $\Omega$ is non-parabolic. If there exists a cone of positive aperture with Neumann condition on both sides, then for $t >1,$
 $$ c_o(t\log^2 t)^{-1} \le  p(t,o,o)\le C_o(t\log^2 t)^{-1} .$$
 
Let $U_1, U_2, U_3$ denote the ends of $\Omega$ with respect to the closure of the unit disk, and let $V_1, V_2, V_3$ denote the actual cones.  Consider the map $A: \{ V_1, V_2, V_3\}$ given by
\[ V_i \mapsto \begin{cases} \frac{3}{2}, & \alpha_i = 0 \\ 1 + \frac{\pi}{\alpha_i}, & \alpha_i >0 \text{ and the cone has Dirichlet boundary condition} \\ 1+ \frac{\pi}{2\alpha_i}, & \alpha_i>0 \text{ and the cone has both boundary conditions}  \end{cases} \]
for $i=1,2,3.$

Then for $t>1,$
\[ c_0 t^{-a} \leq p(t,o,o) \leq C_o t^{-a},\]
where 
\[ a = \min\{ A(V_1), A(V_2), A(V_3)\}.\]

This naturally generalizes for any finite number of cones. Furthermore, the requirement $\sum_{i=1}^3 \alpha_i < 2\pi$ can be removed by moving the vertices of the cones farther away from the origin so that each cone takes up less arc length of the unit circle or by noting that the resulting manifold need not be embedded in the plane.

With slightly more information, we can give more precise estimates on $p(t,x,y)$ for any $t > 1, \ x,y \in \Omega.$ For simplicity of notation, we will assume all cones have positive aperture and Dirichlet boundary condition on both sides. Let $\phi_i$ denote the angle between the positive $x$-axis and the edge of the cone such that when continuing counter-clockwise from this edge, we lie inside of the cone $V_i, \ 1 \leq i \leq 3.$ (The the other edge of the cone is at angle $\phi_i + \alpha_i$ as measured from the positive $x$-axis; note $\phi_i$ maybe be negative.) In polar coordinates, the profile of a cone with aperture $\alpha$ and edge at angle $\phi$ as above with Dirichlet boundary condition on both sides is given by $h_{cone}(r, \theta) = r^{\pi/\alpha} \, \sin\big(\frac{\pi}{\alpha}(\theta - \phi)\big).$ (In the case of cones with Dirichlet condition on one side and Neumann condition on the other side, $h_{cone}(r, \theta) = r^{\pi/2 \alpha} \, \sin\big(\frac{\pi}{2 \alpha}(\theta - \phi)\big)$ if the first edge has Dirichlet condition and a similar formula holds if the first edge instead carries Neumann condition.) The desired estimate depends on whether or not the points $x,y$ lie in the same end $U_i.$ 

Let $x = (|x|, \theta_x), y = (|y|, \theta_y)$ denote $x,y$ written in polar coordinates. Previously, $|x|$ was defined as $\sup_{y \in K} d(x,y)$ to be bounded below away from zero. Below, taking $|x| = d(0,x)$ as is needed for polar coordinates will not be a problem since in all such instances the point $x$ lies in $U_i, \ 1 \leq i \leq 3,$ and hence $|x| \geq 1.$ Above, we have already seen what occurs if both points lie in the middle (and away from any Dirichlet boundary) by examining $p(t,o,o).$ We have the following further cases, where we continue to assume $t>1$:

 \noindent\emph{Case 1:} Suppose $x$ and $y$ are in different ends; without loss of generality assume $x \in U_1, y \in U_2.$ Then 
\begin{align*}
p(t,&x,y) \approx \sin\big(\frac{\pi}{\alpha_1}(\theta_x - \phi_1)\big) \sin\big(\frac{\pi}{\alpha_2}(\theta_y - \phi_2)\big) \cdot \\ &\Bigg[ \frac{1}{ t^{\frac{\pi}{\alpha_*} + 1} \, |x|^{\frac{\pi}{\alpha_1}} \, |y|^{\frac{\pi}{\alpha_2}}} + \frac{|x|^{\frac{\pi}{\alpha_1}}}{t^{\frac{\pi}{\alpha_1} +1} \, |y|^{\frac{\pi}{\alpha_2}}} + \frac{|y|^{\frac{\pi}{\alpha_2}}}{t^{\frac{\pi}{\alpha_2} + 1} \, |x|^{\frac{\pi}{\alpha_1}}} \Bigg] \exp\Big(-c\, \frac{d_+^2(x,y)}{t}\Big),
\end{align*}
where $\alpha^* = \max_{1\leq i \leq 3} \alpha_i.$
For fixed $x,y,$ we obtain the same decay rate as above for $p(t,o,o),$ and if $|x| \approx |y| \approx \sqrt{t},$ then $p(t,x,y)$ decays like 
\[\sin\Big(\frac{\pi}{\alpha_1}(\theta_x - \phi_1)\Big) \sin\Big(\frac{\pi}{\alpha_2}(\theta_y - \phi_2)\Big) t^{-\frac{\pi}{2\alpha_1} -\frac{\pi}{2\alpha_2}-1}.\]

\noindent \emph{Case 2:} Suppose $x, y$ are in the same end, $U_1.$ Then
\begin{align*}
p(&t,x,y)\approx  \sin\Big(\frac{\pi}{\alpha_1}(\theta_x - \phi_1)\Big)  \sin\Big(\frac{\pi}{\alpha_1}(\theta_y - \phi_1)\Big) \cdot \\ 
&\phantom{s} \Bigg(
\frac{|x|^{\frac{\pi}{\alpha_1}}|y|^{\frac{\pi}{\alpha_1}}}{t \,  h(x_{\sqrt{t}})h(y_{\sqrt{t}})} \exp\Big(-c \frac{d_\emptyset^2(x,y)}{t}\Big)  \\
&\hphantom{s}+ 
\bigg[ \frac{1}{t^{\frac{\pi}{\alpha_*}+1} \, |x|^{\frac{\pi}{\alpha_1}} |y|^{\frac{\pi}{\alpha_1}}} +
\frac{|x|^{\frac{\pi}{\alpha_1}}}{t^{\frac{\pi}{\alpha_1}+1}\, |y|^{\frac{\pi}{\alpha_1}}} + \frac{|y|^{\frac{\pi}{\alpha_1}}}{t^{\frac{\pi}{\alpha_1}+1} \, |x|^{\frac{\pi}{\alpha_1}}} \bigg] 
\exp\Big(-c \, \frac{d_+^2(x,y)}{t}\Big)
\Bigg).
\end{align*}
Further, we can compute the quantity $h(x_{\sqrt{t}})$ described following the proof of Theorem \ref{main_thm}. For any $x \in U_1,$
\[ h(x_{\sqrt{t}}) \approx \begin{cases} t^{\frac{\pi}{\alpha_1}}, & \text{ if } 1 \leq |x|^2 \leq t \\ |x|^{\frac{2\pi}{\alpha_1}} \sin^2\Big(\frac{\pi}{\alpha_1}\big(\theta_x + \frac{\sqrt{t}}{|x|}-\phi_1\big))\Big), & \text{ if } 1 \leq t \leq |x|^2.\end{cases}\]

\noindent \emph{Case 3:} Suppose one point lies in the middle; assume this point is $o,$ the origin. The other point $x$ lies in some end, say $U_1.$ Then
\begin{align*}
p(t,o,x) \approx \sin\Big(\frac{\pi}{\alpha_1}(\theta_x - \phi_1)\Big) \bigg[ \frac{1}{t^{\frac{\pi}{\alpha_*}+1}\, |x|^{\frac{\pi}{\alpha_1}}} + \frac{|x|^{\frac{\pi}{\alpha_1}}}{t^{\frac{\pi}{\alpha_1}+1}} \bigg] \exp\Big(-c \frac{|x|^2}{t}\Big).
\end{align*}

\end{example}

\begin{example}
The previous example of unions of cones can also be considered in dimensions other than two. In general, a cone is a subset of $\mathbb{R}^n$ of the form $U = \mathbb{R}_+ \times \Sigma,$ where $\Sigma$ is a subset of $\mathbb{S}^{n-1},$ the $(n-1)$-dimensional unit sphere. If $\Sigma$ has smooth boundary, then $U$ is uniform in $\mathbb{R}^n.$

The profile for such a cone $U$ with Dirichlet boundary condition everywhere (see \cite{BSm, GyS}) is given by 
\[ h_U(x) = |x|^\alpha \phi\bigg(\frac{x}{|x|}\bigg),\]
where $\lambda$ is the first Dirichlet eigenvalue of the spherical Laplacian, $\phi$ is its corresponding eigenfunction, and 
\begin{equation}\label{alpha} \alpha = \frac{\sqrt{(n-2)^2 + 4 \lambda} - (n-2)}{2}.\end{equation}

If we take a union of such cones with Dirichlet boundary condition everywhere, smoothing corners as necessary, then Theorem \ref{main_thm} applies. Everything is as in the previous two-dimensional example, except we may now be unable to compute $\alpha$ and $\phi.$

In particular, consider a union of $k$ such cones, all with Dirichlet boundary condition. Define a map $A$ on the ends $U_1, \dots, U_k$ corresponding to the cones by $A(U_i) = n/2+ \alpha_i,$ where $\alpha_i$ is given by (\ref{alpha}) and indicates the power of $|x|$ appearing in $h_{U_i}$. Then, as above, for all $t>1$,
\[ c_0 t^{-a} \leq p(t,o,o) \leq C_0 t^{-a},\]
where $a = \min \{ A(U_1), \dots A(U_k)\}.$ Here $o$ is a fixed point in $M$ and the constants $c_0,C_0$ depends on $o$. Theorem \ref{main_thm} gives a two-sided estimate over all $t>1$ and $x,y\in \Omega$, but it is more complicated to write down explicitly.

We can also consider the case $n \geq 3$ where at least one of the cones, say $U_1,$ carries Neumann boundary condition instead of Dirichlet boundary condition. Then $h_{U_1} \approx 1,$ and, for any fixed $o$, there are constants $c_1,C_1$ such that, for all $t>1$, we have
\[ c_1 t^{-n/2} \leq p(t,o,o) \leq C_1 t^{-n/2}.\]
Note this is the same behavior as for $p_{\mathbb{R}^n}(t,o,o).$

\end{example}

\begin{example}
Consider the three-dimensional body given in Figure \ref{fig1} with some Dirichlet boundary condition. With Neumann condition everywhere, this figure was considered by Grigor'yan and Saloff-Coste \cite[Example 6.15]{GS5}. If $p_N(t,x,y)$ denotes the heat kernel for this figure with Neumann boundary (the $N$ in $p_N$ stands for Neumann) everywhere and $o$ is a fixed point, then, for $t>1$,
\[ c_0 (t \log^2 t)^{-1} \leq p_N(t,o,o) \leq  C_0 (t \log^2 t)^{-1}.\]

The most natural place to add Dirichlet boundary condition is on the three dimensional cone. The cone with Dirichlet boundary everywhere has a profile with growth of power $\alpha >0$ by the previous example, so that the volume of the cone weighted by its profile is approximately $r^{2\alpha +3}.$ Thus the volume in the cone grows faster than $r^2 \log^2 r,$ which describes how volume grows after $h$-transform in the infinite solid disk. Hence $p_D(t,o,o)$ has the same long-term decay in time as $p_N(t,o,o).$

In fact, the previous paragraph still holds true when we impose \emph{any} Dirichlet boundary condition on the cone in such a way that condition (*) holds, as the following lemma demonstrates that profiles cannot decrease volume in some sense.

\end{example}

\begin{lem}
Let $(U, \sigma)$ be an unbounded weighted Riemannian manifold that is uniform in its closure $\overline{U},$ which is a Harnack manifold. Let $u$ denote the profile for $U.$ Then there exists $C>0$ such that 
\[ \int_{B_U(x,r)} u^2(y)\sigma(y) dy =: V_{u}(x,r) \geq C V(x,r) \]
for all $x \in \overline{U}$ and all $r>0$ sufficiently large (where $r$ may depend on $x$).
\end{lem}

\begin{proof}
It is not possible that $u(x) \to 0$ as $x \to \infty$ since if this were the case, the maximum principle combined with the Hopf boundary lemma imply $u\equiv 0.$ Therefore there exists a sequence of points $\{ z_j \}_{j=1}^\infty$ in $U$ and a number $\varepsilon >0$ such that for any fixed point $o \in U,$ $d(z_j, o) \to \infty$ as $j \to \infty$, and $u(z_j) \geq \varepsilon.$

Then by Theorem \ref{HarnackEnds} there exist constants $0< c_0 < +\infty$ and $0 < c_1 \leq C_1 < +\infty$ such that 
\[ c_1 u(x_r)^2 V(x,r) \leq V_{u}(x,r) \leq C_1 u(x_r)^2 V(x,r) \]
for all $x \in \overline{U}, r>0,$ and $x_r$ such that $d(x,x_r) \geq r/4$ and $d(x_r, \partial U) \geq c_0 r/8.$ Moreover, by the proof of Theorem 4.7 in \cite{GyS}, there exists a constant $0<C_2 < +\infty$ such that 
\[u(y) \leq C_2 u(x_r) \quad \forall x \in \overline{U},\ y \in B(x,r), \ r>0.\]

Given $x\in \overline{U},$ let $r>0$ be sufficiently large so that $z_j \in B(x,r)$ for some $j=1,2,\dots.$ Then 
\[ V_{u}(x,r) \geq \frac{c_1}{C_2} u(z_j)^2 V(x,r) \geq \frac{c_1 \varepsilon^2}{C_2} V(x,r)\]
as claimed.

\end{proof}

\begin{example}\label{2parabolas}
Consider two copies of the exterior of a parabola in $\mathbb{R}^2$. Impose Dirichlet condition along each parabola and glue the two copies via a collar, as in Figure \ref{ext_para}. If $K$ indicates the compact collar, then this manifold $\Omega$ has two ends, both of which are the exterior of a parabola in $\mathbb{R}^2,$ minus a disk. 

\begin{figure}[h]\begin{center}
\begin{tikzpicture}[scale=.8] 

\draw[ultra thick] (1,2)--(0,0)--(4,-.5)--(5,1.5);

\draw[blue, ultra thick] (1,2) .. controls (5/2,3/4) .. (5,1.5);
\draw [thick] (2.25,.3) circle (.4);

\draw [ultra thick] (1,-1)--(0,-3)--(4,-3.5)--(5,-1.5);
\draw[ultra thick, blue] (1,-1) .. controls (5/2,-9/4) .. (5,-1.5);

\filldraw[dashed, thick, gray, opacity=.4] (2.25,-2.7) circle (.4);
\draw[thick] (1.85,-2.7) arc (180:360:.4);

\draw [thick] (1.85,.3)--(1.85,-2.7);
\draw [thick] (2.65,.3)--(2.65,-2.7);

\path [fill, gray, opacity =.4] (1.85,.3)--(1.85,-2.7) arc (180:0:.4)--(2.65,.3) arc (0:-180:.4);

\node at (.6,.3) {$\Omega$};

\end{tikzpicture}\end{center}\caption{A connected sum of the exterior of two parabolas in $\mathbb{R}^2.$}\label{ext_para}\end{figure}
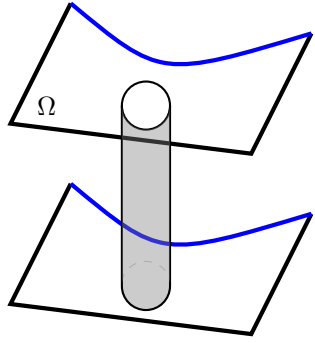

As $\mathbb{R}^2$ minus a parabola is the complement of a convex set, it is uniform in its closure \cite[Proposition 6.16]{GyS}, and removing a disk of fixed radius will not change this property. Further, with Neumann condition along both parabola and disk, this manifold is Harnack \cite[Theorem 3.10]{GyS}. Thus hypotheses (H1) and (H2) are satisfied. 

The global profile $h$ for $\Omega$ will behave like the profile for $\mathbb{R}^2$ minus a parabola and a disk in each end. Denote the profile for the exterior of a parabola by $h_{EP}$. Consider the exterior of the parabola weighted by $h_{EP}^2.$ Then this space is transient as satisfies the parabolic Harnack inequality, so removing the disk, a compact set, has little effect \cite{GS3, GyS}. What is important to us here is that $\hat{h},$ the profile for $\mathbb{R}^2$ minus a parabola and a disk, weighted by $h_{EP}^2,$ is essentially constant away from the disk. As the profile for the ends we are interested in is the product of $h_{EP}$ and $\hat{h},$ it behaves like $h_{EP}$ when away from the disk. Thus the global profile $h$ for $\Omega,$ which appears in Theorem \ref{main_thm}, also behaves like $h_{EP}.$

The profile for the exterior of the parabola $x_2 = x_1^2$ in $\mathbb{R}^2,$ that is, the space $EP = \{ (x_1, x_2) \in \mathbb{R}^2 : x_2 < x_1^2\},$ is given by 
\begin{equation}\label{h_extpara} h_{EP}(x) = \sqrt{2\Big(\sqrt{x_1^2 +(1/4 - x_2)^2}+1/4-x_2\Big)} -1 .\end{equation}
The profile for the exterior of any parabola can be found by making an appropriate change of variables in this formula. 

Using (\ref{h_extpara}) to compute quantities appearing in Theorem \ref{main_thm}, for any fixed point $o,$ there exist constants $0<c \leq C < +\infty$ such that for all $t>1,$
\[ c t^{-3/2}\leq p(t,o,o) \leq C t^{-3/2}.\]

Now fix $0<r_1 \leq r_2.$ Assume that $x,y$ lie in different copies of the exterior of the parabola and are both at distance approximately $\sqrt{t}$ from the collar, that is, $r_1 \sqrt{t} \leq |x|, |y| \leq r_2 \sqrt{t}.$ Since $V_{i,h}(r)$ satisfies (\ref{vol_cond}),
\[ H_h(x,t)\approx \frac{|x|^2}{V_{i_x, h}(|x|)} \approx \frac{t}{V_{i_x,h}(\sqrt{t})} \approx t^{-1/2}. \]
Likewise, $H_h(y, t) \approx t^{-1/2}.$ Thus there exists constants $0<c_1 \leq C_1 < +\infty$ such that, for $t$ sufficiently large and all such $x,y$,
\[ c_1 h(x)h(y) t^{-2} \leq p(t,x,y) \leq C_1 h(x)h(y) t^{-2}.\]
Depending on the location of $x,y$ relative to the parabola, $h(x), \ h(y)$ can range from zero to behaving like $t^{1/4}.$ See Figure \ref{para_circle}. Note if $h(x)\approx h(y) \approx t^{1/4},$ then $p(t,x,y) \approx t^{-3/2}.$

\begin{figure}[h]\begin{center}
\begin{tikzpicture}[domain=0:5, scale=.5] 

\draw[<->, blue, ultra thick]   plot[smooth,domain=-3:3] (\x, {\x*\x});
\draw[thick] (0,-1.5) circle(.5);

\draw (0, .25) circle(8);
\draw[green,ultra thick] (-8,.25) arc (180:360:8);

\draw[<->] (-9,0)--(9,0);
\draw[<->] (0,9)--(0,-9);

\node at (8.6,0) [anchor=north] {$x_1$};
\node at (0,8.6) [anchor=east] {$x_2$};
\node at (2.2,3.8) [anchor=west] {$x_2 = x_1^2$};

\draw[fill=black] (0,.25) circle (.05);
\node at (0,.9) [anchor=east] {\small{$\frac{1}{4}$}};
\draw[thick, ->] (0,.25)--(-8,.25);
\node at (-5,.25) [anchor=south] {$r\sqrt{t}$};

\draw[<-,thick] (.2,-1.96)--(.4,-2.5);
\node at (.3,-2.5) [anchor=west] {collar};

\draw[->,thick] (-4,8)--(-2.78,7.75);
\node at (-4,8) [anchor=east]{$h(x)=0$};
\draw[->,thick] (4,8)--(2.78,7.75);
\node at (4,8) [anchor=west]{$h(x)=0$};

\node at (-2.8,-8.2) {$h(x) \approx t^{1/4}$};

\end{tikzpicture}\end{center}\caption{If $|x| \approx \sqrt{t}$, then for $t$ sufficiently large, $x$ is also approximately at distance $\sqrt{t}$ from the focus of the parabola; denote this distance by $r\sqrt{t}$ so that $x$ lies on the circle depicted above. In the bottom half of this circle, colored green, $h(x) \approx t^{1/4}$ for large $t$. As $x$ travels along the circle toward the parabola, $h(x)$ decreases to zero.}\label{para_circle}\end{figure}
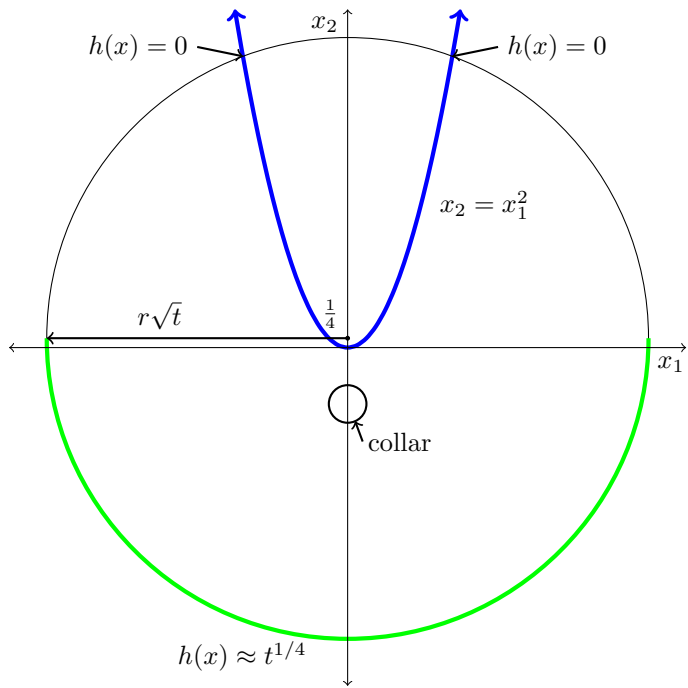

\end{example}

\begin{example}\label{parabola_plane}
Consider Example \ref{2parabolas}, except remove a parabola with Dirichlet condition from only \emph{one} plane. Then the profile $h(x)$ for this manifold behaves like $h_{EP}(x),$ which is given by (\ref{h_extpara}), in the end with the parabola removed, and, in the plane without the parabola removed, $h(x)$ behaves like $\log(|x|),$ as this is the harmonic profile for the plane minus a disk. Thus, for $o$ fixed, the presence of the plane without the parabola removed results in heat kernel decay of the form
\[ c(t \log^2 t)^{-1} \leq p(t,o,o) \leq C(t\log^2 t)^{-1} \]
for all $t>1.$

Again, fix $0<r_1 \leq r_2$ and take $x$ in the plane minus the parabola and $y$ in the plane, both so that their distance to the collar lies between $r_1\sqrt{t}$ and $r_2 \sqrt{t}.$ We still have $H(x,t) \approx t^{-1/2}$ as in the previous example, but $V_{i_y,h}$ does not satisfy (\ref{vol_cond}) for $y$ in the plane without the parabola removed. Working directly with the definition of $H_h(y,t)$ (see formula (\ref{H_h})), we find $H_h(y,t) \approx (\log^2 t)^{-1}$ for all such $y.$ Therefore, for $t>1,$ we obtain the estimate 
\[ c_1 h(x) (t^{3/2}\log t)^{-1} \leq p(t,x,y) \leq C_1 h(x) (t^{3/2}\log t)^{-1}.\]
Again, the behavior of $h(x)$ depends on where $x$ is relative to the parabola as in Figure \ref{para_circle}. If $h(x) \approx t^{1/4},$ then $p(t,x,y) \approx (t^{5/4} \log t)^{-1}.$

\end{example}

\begin{example}\label{paraboloid_ex}
Now consider an analog of Example \ref{2parabolas} or \ref{parabola_plane}, but in higher dimensions. For instance, remove a paraboloid of revolution from a copy of $\mathbb{R}^3$ and impose Dirichlet boundary condition on the resulting boundary. Take two copies of this space and glue them via a collar. Then, in theory, we may apply the technique of the previous two examples to estimate the heat kernel of this space. However, estimates for the profile of $\mathbb{R}^3$ minus a paraboloid are not known. Thus, in practice, we cannot compute explicit decay rates of this heat kernel. 
\end{example}

\begin{appendices}

\section{Generalities and notation}\label{notes}
Let $(\Omega, \delta \Omega)$ be a weighted Riemannian manifold with boundary.  If the associated metric space $(\Omega,d)$ is not complete, let $\widetilde{\Omega}$ be its completion and $\partial \Omega=\widetilde{\Omega}\setminus \Omega$.  Note that this setup is more general than the special case considered in the main part of the paper where condition (*) holds:
\begin{itemize}
\item[(*)] 
$\Omega$ is a submanifold of the weighted complete Riemannian manifold with boundary $(M, \delta M)$ with $M\setminus \Omega\subseteq \delta  M$, and $\partial \Omega$ has countably many connected components such that every point in $M$ has a neighborhood containing only finitely many of these connected components, and these components are themselves codimension $1$ submanifolds (with boundary) of $M$.
\end{itemize}
Indeed, in general, $\widetilde{\Omega}$ may not be a manifold with boundary. One more subtle difference of importance to us here is that the weight $\sigma$ on $\Omega$, which  is smooth and positive in $\Omega$, might have a variety of behaviors when one approaches the boundary $\partial \Omega$.  Under hypothesis (*), the weight $\sigma$ is smooth and positive up to the boundary $\delta M$.

Recall that $\Omega^\bullet =\Omega\setminus \delta\Omega, \ \Wloc(\Omega^\bullet)$ is the local Sobolev space on $\Omega^\bullet$, and $W^1_0(\Omega)$ is the closure of $\mathcal C^\infty_c(\Omega)$ under the norm $\left(\int_{\Omega}|f|^2d\mu+\int_\Omega |\nabla f|^2 d\mu\right)^{1/2}$ (Definition \ref{localSobolev}). For $U\subset \Omega,$ the space $W^1(U)$ is the set of functions $f \in \Wloc(U^\bullet)$ such that 
\[ \left(\int_U |f|^2d\mu+\int_U |\nabla f|^2 d\mu\right)^{1/2} < + \infty.\]
Then the Sobolev space $\Wloc(U)$ is defined as the set of functions where for any open relatively compact $V \subset U,$ there is a function $f^V \in W^1(U)$ such that $f = f^V$ almost everywhere on $V.$ We also let $\mbox{Lip}(V)$ be the space of bounded Lipschitz functions on $V$. 

 Recall the following definition used in \cite{GS5}:
\begin{defin}[Relatively connected annuli property]\label{RCA}\ A metric space $(M, d)$ satisfies the relatively connected annuli property ((RCA), for short) with respect to a point $o\in M$ if there exists a constant $C_A$ such that for any $r\ge C_A^2$  and all $x,y \in M$ such that $d(o,x) = d(o,y) = r$, there exists a continuous path $\gamma : [0,1]\to  M$  with $\gamma(0) = x, \gamma(1) = y$ whose image is contained in $B(o, C_Ar)\setminus B(o, C_A^{-1}r)$.\end{defin}

\section{Local and global harmonic functions}\label{harmonic}

Throughout, we choose to use appropriate weak definitions of solutions of the Laplace or heat equation even though,  in the special set-up of interest to us, because of  various simplifying hypotheses made in the main parts of the article, such solutions are, in fact, classical solutions, including with respect to the boundary conditions (see, for instance, \cite{Lieber}).

\begin{defin}[Harmonic function in an open set $U$ of $\Omega$] Let $U$ be an open subset of $\Omega$. A function $u$ defined in $U$ is  a (local)  harmonic function in $U$  if 
$u\in W_{\mbox{\tiny loc}}(U)$ and, for any $\phi\in \mathcal C^\infty_c(U)$, 
$$\int_Ug(\nabla u,\nabla \phi)d\mu=0.$$
In classical terms, $u \in \mathcal C^\infty _{\mbox{\tiny loc}}(U)$,  $\Delta u=0$ in $U\cap \Omega^\bullet=U\cap M^\bullet$, and  $u$ has vanishing normal derivative along
$U\cap \delta \Omega$ (and no condition along $\overline{U}\cap \partial \Omega$).\end{defin}

\begin{defin}[Harmonic function in an open set $U\subset \Omega$ vanishing along $\partial \Omega$] \label{def-vanish}
Let $U$ be an open subset of $\Omega$.
A function $u$ defined in $U$ is  a (local)  harmonic function in $U$ with Dirichlet boundary condition along $\partial \Omega$ (i.e., vanishing along $\partial \Omega$) if $u$ is locally harmonic in $U$ and, for any $\psi\in \mathcal C_c(U^\sharp)\cap \mbox{Lip}(U^{\sharp})$,
$u\psi\in W^1_0(U)$.  Here $U^\sharp$ is the largest open set in $\widetilde{\Omega} $ such that $U^\sharp\cap \Omega=U$. In classical terms, under condition (*), $u \in \mathcal C^\infty _{\mbox{\tiny loc}}(U)$,  $\Delta u=0$ in $U\cap \Omega^\bullet=U\cap M^\bullet$,  $u$ has vanishing normal derivative along
$U\cap \delta \Omega$, and $u$ can be extended continuously by setting $u(x)=0$ at any point $x\in \partial \Omega$ which is at positive distance from 
$\Omega \setminus U$. \end{defin}

\begin{defin}[Global harmonic function in $\Omega$]\label{def-globalharm} A global harmonic function in $\Omega$ is a function $u$ in $\Omega$ which is locally harmonic in $\Omega$ and vanishes along $\partial \Omega$. \end{defin}

\begin{rem} This last definition applies to the case $M=\Omega,$ providing the definition of global harmonic function in $M$. In that case, there is no Dirichlet boundary condition as $\partial M=\emptyset$.
\end{rem}

\begin{defin}[Elliptic Harnack inequality] \label{EHI}We say that:
\begin{itemize} 
\item The elliptic Harnack  inequality holds locally in a  subset $V$ of $\Omega$  if for any compact set $K\subset V$ there exist $H_K$ and $r_K>0$ such that,
for all $(x,r)\in K\times (0,r_K)$, and any positive  harmonic function $u$ in $B_\Omega(x,2r),$
$$\sup_{B_\Omega(x,r)}\{u\}\le H_K\inf_{B_\Omega(x,r)}\{u\}.$$
\item  The elliptic Harnack inequality holds up to scale $r_0$ over a  subset $K$ of $\Omega$  if there is a constant $H_{K,r_0}$ such that,
for all $ (x,r)\in K\times (0,r_0)$ and any positive harmonic function $u$ in $B_\Omega(x,2r),$
$$\sup_{B_\Omega(x,r)}\{u\}\le H_{K,r_0}\inf_{B_\Omega(x,r)}\{u\}.$$
\item The elliptic Harnack  inequality holds uniformly in an open subset $U$ of $\Omega$ if there is a constant $H_U$ such that
for all $ (x,r)\in U\times (0,+\infty)$  such that $B_\Omega(x,2r)\subset U$ and any positive harmonic function $u$ in $B_\Omega (x,2r)$,
$$\sup_{B_\Omega (x,r)}\{u\}\le H_U\inf_{B_\Omega (x,r)}\{u\}.$$
\end{itemize}
\end{defin}

\begin{rem}The elliptic Harnack inequality always holds locally on $\Omega$. It holds up to scale $r_0=d(K,\partial \Omega)>0$ on any compact subset $K$ of $\Omega$.  
  Under assumption (*),   the elliptic Harnack inequality always holds locally on $\Omega$ and on $M.$ (They do not mean the same thing.) It also holds up to scale $r_0$ for any fixed $r_0$ on any compact subset of $M$. \end{rem}
  
\begin{defin}[Boundary elliptic Harnack inequality] \label{EBHI}These definitions are only useful when $\partial\Omega\neq \emptyset$.

\begin{itemize} 
\item  The boundary elliptic Harnack  inequality holds locally on a subset $V$ of $\partial \Omega$  if for any compact set $K\subset V$ there exist $H_K$ and $r_K>0$ such that,
for all $(x,r)\in K\times (0,r_K)$, and any two positive harmonic functions $u,v$ in $B_\Omega(x,2r)=\Omega\cap B_{\widetilde{\Omega}}(x,2r)$ vanishing along $\partial \Omega $,
$$\sup_{B_\Omega(x,r)}\{u/v\}\le H_K\inf_{B_\Omega (x,r)}\{u/v\}.$$
\item  The boundary elliptic Harnack inequality holds up to scale $r_0$ over a  subset $K$ of $\partial \Omega $  if there is a constant $H_{K,r_0}$ such that,
for all $ (x,r)\in K\times (0,r_0)$ and any two positive harmonic functions $u,v$ in $B_\Omega(x,2r)=\Omega\cap B_{\widetilde{\Omega}}(x,2r)$ vanishing along $\partial  \Omega,$
$$\sup_{B(x,r)}\{u/v\}\le H_{K,r_0}\inf_{B(x,r)}\{u/v\}.$$
\item The boundary elliptic Harnack  inequality holds uniformly in an open subset $U$ of $\Omega$  if there is a constant $H_U$ such that
for all $ (x,r)\in \partial \Omega \times (0,\infty)$  such that $B(x,2r)\subset U$ and any  two positive harmonic functions $u,v$ in $B(x,2r)$ vanishing along $\partial \Omega,$
$$\sup_{B(x,r)}\{u/v\}\le H_U\inf_{B(x,r)}\{u/v\}.$$
\end{itemize}
\end{defin}

\begin{rem} Whether a boundary  elliptic Harnack inequality holds or not depends on the nature of the boundary $\partial \Omega$. Under assumption (*), the boundary elliptic Harnack inequality always holds locally on $\Omega$ and, for each  $r_0>0$, up to scale $r_0$ on any compact
subset $K$ of $\partial \Omega$.  \end{rem}

\section{Local and global solutions of the heat equation}

To save space, we refer the reader to \cite{ESC, GyS, St1,St2,St3} for the definition of local weak solutions of the heat equation in an open cylindrical domain  $(a,b)\times U \subset \mathbb R\times \Omega$, in the context of the
strictly local regular Dirichlet space $(W^1_0(\Omega), \int_\Omega g(\nabla f,\nabla f)d\mu)$. Such weak solutions are automatically smooth in time, so one can be a bit cavalier  with the details of such definitions. In fact, these local weak solutions are always smooth in 
$(a,b)\times U,$ including up to $\delta\Omega \cap U$ where they satisfy the Neumann boundary condition.

For the definition of weak solutions in an open set $U$ of $\Omega$ vanishing along $\partial \Omega$, we refer the reader to \cite{GyS}. Simply put, given that weak solutions are smooth in time and at any point in $U$, the condition
that the solution vanishes along the relevant part of $\partial \Omega$ can be captured as in Definition \ref{def-vanish} by requiring that, for any $t\in (a,b)$ and any $\psi\in \mathcal C_c(U^\sharp)\cap \mbox{Lip}(U^\sharp)$, $u\psi\in W^1_0(U)$. In fact, under condition (*), such a solution will vanish continuously along the relevant part of $\partial \Omega$ \cite{Lieber}.

\begin{defin}[Global solution of the heat equation in $(a,b)\times\Omega$] A global solution of the heat equation function in $(a,b)\times \Omega$ is a function $u$ in $(a,b)\times M$ which is smooth in $(a,b)\times \Omega$, satisfies $(\partial _t -\Delta)u=0$ in $(a,b)\times M^\bullet$, has vanishing normal derivative on $\delta \Omega$  and vanishes along $\partial \Omega$. \end{defin}

 Given a time-space cylinder $Q=(s,s+ 4r^2)\times B(x,2r)$, set $Q_-=(s+r^2,s+2r^2)\times B(x,r)$ and $Q_+=(s+3r^2,s+4r^2)\times B(x,r)$.

\begin{defin}[Parabolic Harnack inequality]\label{PHI} We say that:
\begin{itemize} 
\item  The parabolic Harnack  inequality holds locally in a subset $V$ of $\Omega$  if for any compact set $K\subset U$ there exist $H_K$ and $r_K>0$ such that,
for all $s\in \mathbb R$, $(x,r)\in K\times (0,r_K)$, and any local solution $u\ge 0$ of the heat equation in $Q=(s, s+ 4r^2)\times  B_\Omega(x,2r),$
$$\sup_{Q_-}\{u\}\le H_K\inf_{Q_+}\{u\}.$$
\item  The parabolic Harnack inequality holds up to scale $r_0$ over a  subset $K$ of $\Omega$  if there is a constant $H_{K,r_0}$ such that,
for all $s\in \mathbb R$, $ (x,r)\in K\times (0,r_0)$ and any local solution $u\ge 0$ of the heat equation in $Q=(s,s+4r^2)\times B_\Omega(x,2r)$,
$$\sup_{Q_-}\{u\}\le H_{K,r_0}\inf_{Q_+}\{u\}.$$
\item The parabolic Harnack  inequality holds uniformly in an open subset $U$ of $\Omega$ if there is a constant $H_U$ such that,
for all $s\in \mathbb R$ and $ (x,r)\in U\times (0,+\infty)$  such that $B(x,2r)\subset U$ and any local solution $u\geq0$ of the heat equation in $(s,s+4r^2)\times B_\Omega (x,2r)$,
$$\sup_{Q_-}\{u\}\le H_U\inf_{Q_+}\{u\}.$$
\end{itemize}
\end{defin}

\section{Doubling and Poincar\'e}

\begin{defin}[Doubling]\label{VD} Very generally, doubling refers to  the volume function property that  $V(x,2r)\le C V(x,r)$ where $(x,r)$ belong to some specific subset of $\Omega\times (0,+\infty)$.
\begin{itemize} 
\item A set $V$ is locally doubling if for any compact set $K\subset V$ there exists $r_0(K)>0$ such that
$K$ is doubling up to scale $r_0(K)$.
\item  An arbitrary set $K$ is  doubling up to scale $r_0$ if there is a constant $C_{K,r_0}$ such that
for all $ (x,r)\in K\times (0,r_0)$, $V(x,r)\le C_{K,r_0} V(x,2r)$.
\item An open subset $U$ of $\Omega$ or $\widetilde{\Omega}$ is uniformly doubling (or doubling for short) if there is a constant $C_U$ such that
for all $ (x,r)\in U\times (0,+\infty)$ such that $B(x,2r)\subset U$, $V(x,2r)\le C_U V(x,r)$.\end{itemize}
\end{defin}

\begin{rem} A manifold with boundary is always locally doubling. It may or not be doubling up to scale $r_0$ for some $r_0>0$. It may or not be uniformly doubling.   Euclidean space $\mathbb R^n$  is doubling, as is any complete Riemannian manifold without boundary with non-negative Ricci curvature. 
Convex domains in $\mathbb R^n$ are doubling.
Hyperbolic space is doubling up to any fixed scale $r_0,$ but it is not doubling. A complete Riemannian manifold without boundary with Ricci curvature bounded below is doubling up to any fixed scale $r_0$.   
\end{rem}

A Poincar\'e inequality  is an inequality of the form
$$\forall f\in \mathcal C^\infty (B(x,r)),\;\;\int_{B(x,r)} |f-f_B|^2d\mu \le P r^2 \int_{B(x,r)}|\nabla f|^2 d\mu,$$
where $f_B$ is the average value of $f$ over $B=B(x,r)$.

\begin{defin}[Poincar\'e inequality]\label{PI} Consider the following three versions:
\begin{itemize} 
\item  The Poincar\'e inequality holds locally in a subset $V$ of $\Omega$ or $\widetilde{\Omega}$  if for any compact set $K\subset V$ there exists $r_0(K)>0$  and a constant $P_K$ such that,
for all $(x,r)\in K\times (0,r_0(K))$,
$$\forall f\in \mathcal C^\infty (B(x,r)),\;\;\int_{B(x,r)} |f-f_B|^2d\mu \le P_K r^2 \int_{B(x,r)}|\nabla f|^2 d\mu.$$
\item  The Poincar\'e inequality holds up to scale $r_0$ over a  subset $K$ of $\Omega$ or $\widetilde{\Omega}$  if there is a constant $P_{K,r_0}$ such that,
for all $ (x,r)\in K\times (0,r_0)$,
$$\forall f\in \mathcal C^\infty (B(x,r)),\;\;\int_{B(x,r)} |f-f_B|^2d\mu \le P_{K,r_0} r^2 \int_{B(x,r)}|\nabla f|^2 d\mu.$$
\item The Poincar\'e inequality holds uniformly in an open subset $U$ of $\Omega$  or $\widetilde{\Omega}$  if there is a constant $P_U$ such that
for all $ (x,r)\in U\times (0,+\infty)$  such that $B(x,r)\subset U$,
$$\forall f\in \mathcal C^\infty (B(x,r)),\;\;\int_{B(x,r)} |f-f_B|^2d\mu \le P_U r^2 \int_{B(x,r)}|\nabla f|^2 d\mu.$$\end{itemize}
\end{defin}

\begin{rem}  A Poincar\'e inequality always holds locally on any manifold with boundary.  A Poincar\'e inequality up to scale $r_0$ for some $r_0>0$ may hold or not on a manifold with boundary. A Poincar\'e inequality may hold uniformly or not on a manifold with boundary.  A Poincar\'e inequality holds uniformly on Euclidean space $\mathbb R^n,$ and it also holds uniformly on any complete Riemannian manifold without boundary with non-negative Ricci curvature. 
A Poincar\'e inequality up to scale $r_0$ for any  fixed $r_0>0$ holds on hyperbolic space, but it does not hold uniformly at all scales. A Poincar\'e inequality up to scale $r_0$ for any fixed $r_0>0$ holds on any complete Riemannian manifold without boundary with bounded Ricci curvature.
\end{rem} 

\section{Harnack weighted manifolds}\label{Harnack}  

As in Appendix \ref{notes}, let $(\Omega,\delta \Omega)$ be a Riemannian manifold. We do not assume it is complete. Let $\widetilde{\Omega}$ be its metric completion and $\partial \Omega =\widetilde{\Omega}\setminus \Omega$.  Let $\sigma$ be a smooth positive weight on $\Omega$.   
We consider the (local regular) Dirichlet space $(W^1_0(\Omega),\int_\Omega |\nabla f|^2 d\mu)$ and the associated heat equation (see, e.g.,  \cite{GyS, St2,St3} for details). 

\begin{defin}[Harnack manifold] \label{HarnackManifold}We say that a weighted Riemannian manifold  $\Omega$ is a Harnack manifold if the parabolic Harnack inequality holds uniformly in $\Omega$.
\end{defin}  
Under relatively mild conditions on $\Omega,\ \widetilde{\Omega},$ and the weight $\sigma$, this condition is known to be equivalent (\cite{Gri,PSHDuke,St3,GyS}) to the validity of the volume doubling condition and Poincar\'e inequality, uniformly in $\widetilde{\Omega}$. It is also equivalent to the validity of the two-sided (Gaussian) heat kernel estimate
\begin{equation}\label{TSG}
\frac{c_1e^{-C_2 \frac{d^2}{t}}}{V(x,\sqrt{t})}\le p_\Omega(t,x,y) \le \frac{C_1 e^{-c_2 \frac{d^2}{t}}}{V(x,\sqrt{t})},\;\;d=d(x,y).
\end{equation}

\begin{rem} The best known large class of Harnack manifolds is the class of complete Riemannian manifolds with non-negative Ricci curvature  (see \cite{Asp} and the references therein).  In this case the weight is the constant weight $1$. Reference \cite{GS1} discusses how to obtain examples with non-trivial weights. 
We are interested in the case when $\widetilde{\Omega}$ is a (smooth) manifold satisfying condition (*). In this case, assuming that $\sigma$ has a continuous extension to $\partial \Omega$,
it is necessary for the weight $\sigma$ to vanish at the boundary in order for the weighted manifold $\Omega$ to have a chance to be a Harnack manifold. 
One of the simplest examples of Harnack manifold of this type is the upper-half Euclidean space
$\mathbb R^n_+=\{x=(x_1,\dots,x_n)\in \mathbb R^n: x_n>0\}$ equipped with the weight $\sigma(x)= x_n^2$. See \cite{GyS} for many more examples. 
\end{rem}
 
We will make use of the following key theorems. See \cite{GyS} for a discussion of more general versions of these theorems.

\begin{thm}\label{HarnackNotComplete} Let $(\Omega,\sigma)$ be a weighted Riemannian manifold with boundary. Assume that $\widetilde{\Omega}=M$ is a manifold with boundary and that $\partial \Omega=M\setminus \Omega$ satisfies condition (*). Assume that the weight $\sigma$ has a continuous extension to $M$, vanishing on $\partial \Omega$ and such that the restriction to $\Omega$ of any Lipschitz function  compactly supported in $M$ is in  $W^1_0(\Omega)$.  Then the weighted manifold $(\Omega,\sigma)$ is Harnack if and only if  $(\Omega, \sigma)$ is uniformly doubling and the Poincar\'e inequality holds uniformly.
\end{thm}
This is a slight extension of the results in \cite{Gri,PSHDuke}, which essentially cover the case $\Omega=M$. This extension is contained in the more general Dirichlet space version given in \cite{St3}. 

The following important theorem follows from Section 5 of \cite{GyS}.

\begin{thm}\label{HarnackEnds} Assume that $(M,\sigma)$ is a weighted complete Riemannian manifold which is uniformly Harnack.  Let $\Omega$ be an open subset of $M$ such that $\partial \Omega=M\setminus \Omega$ is a subset of the boundary $\delta M$ and satisfies (*). Assume that $\Omega$ is a uniform subset of $M$ (Definition \ref{uniform}) and let $h$ be a positive harmonic function vanishing along $\partial \Omega$ (a harmonic profile for $\Omega$). 

Then there are constants $0<c\le C<+\infty$ such that, for any $x\in M, r>0$,  and any 
$x_r$ such that $x_r\in B(x,Ar)$ and $d(x_r,\partial \Omega)\ge a r$, we have 
\[ ch(x_r)^2 V(x,r)\leq V_{h}(x,r):=\int_{B(x,r)}  h^2  d\mu\le  Ch(x_r)^2 V(x,r) \]
where, as usual, $V(x,r)=\mu(B(x,r))$.

Moreover, the Riemannian manifold $\Omega$ weighted by $\sigma_h=\sigma h^2$ is a Harnack manifold. In particular, if $p_{\Omega, h^2}$ indicates the heat kernel for $(\Omega, \sigma_h),$ there exist constants $c_1, c_2, c_3, c_4 >0$ such that $\forall t>0, \ x,y \in \Omega$ 
\[ \frac{c_1}{h(x_{\sqrt{t}})^2 V(x, \sqrt{t})} \exp\Big(-\frac{d(x,y)}{c_2t}\Big) \leq p_{\Omega, h^2}(t,x,y)\]
and
\[p_{\Omega, h^2}(t,x,y) \leq \frac{c_3}{h(x_{\sqrt{t}})^2 V(x, \sqrt{t})} \exp\Big(-\frac{d(x,y)}{c_4 t}\Big) .\]
\end{thm}

We also a need an extension of a particular case of the main result of \cite{GS5} that holds on a certain class of manifolds, some of which may be incomplete. 

\begin{thm}\label{HK-ends-estimate}
Let $(\Omega, \sigma)$ be a weighted Riemannian manifold with boundary such that $\widetilde{\Omega} = M$ is a manifold with boundary and $(\Omega, \sigma)$ satisfies (*). Assume that the weight $\sigma$ has a continuous extension to $M,$ vanishing on $\partial \Omega$ and such that the restriction to $\Omega$ of any Lipschitz function with compact support in $M$ belongs to $W_0^1(\Omega)$. If $\Omega$ has ends $U_1, \dots, U_k,$ further assume that each $U_i \cup \inner U_i, \ 1 \leq i \leq k,$ is Harnack in the sense of Theorem \ref{HarnackNotComplete} and non-parabolic (see Appendix \ref{parabolic}). Then for all $x,y \in \Omega$ and $t> 1,$
\begin{align*}
p(t,x,y) &\approx \ C \Bigg[\frac{1}{\sqrt{V_{i_x}(x,\sqrt{t})V_{i_y}(y, \sqrt{t})}} \exp \Big(-c\frac{d^2_\emptyset(x,y)}{t}\Big) \\
&+\Bigg(\frac{H(x,t) H(y,t)}{V_{0}(\sqrt{t})} + \frac{H(y,t)}{V_{i_x}(\sqrt{t})} + \frac{H(x,t)}{V_{i_y}(\sqrt{t})} \Bigg) \, \exp \Big( -c \frac{d_+^2(x,y)}{t}\Big)\Bigg],
\end{align*}
where the constants $C,c$ take different values in the upper and lower bounds.\end{thm}

Here
\[ i_x = \begin{cases} i, & \text{ if } x \in U_i \\ 0, & \text{ if } x \in K,\end{cases}\]
and, so that $|x|$ is bounded below away from zero, we set 
\[ |x| := \sup_{y \in K} d(x,y), \; x \in M.\]
Then if $B_i(x,r)$ denotes a ball in $U_i$ centered at $x$ with radius $r$ and $o_i$ is a fixed reference point on $\inner U_i, \ 1 \leq i \leq k,$
\[V_{i}(r) := V_{i}(o_i,r) = \int_{B_i(o_i,r)} \sigma_i(x) \, dx.\]
We further set $V_{0}(r) = \min_{1 \leq i \leq k} V_{i}(r).$ The notation $d_+(x,y)$ refers to the distance between $x$ and $y$ when passing through the compact middle $K,$ whereas $d_\emptyset(x,y)$ refers to the distance between $x$ and $y$ if we avoid $K.$ Finally, we define
\begin{equation*}
 H(x,t) = \min \Bigg\{ 1, \frac{|x|^2}{V_{i_x}(|x|)} + \Big(\int_{|x|^2}^t \frac{ds}{V_{i_x}(\sqrt{s}\,)}\Big)_+ \Bigg\}.
 \end{equation*}
 
\begin{rem}\label{simplify_H}
If $V_{i_x}(r)$ satisfies the condition that for some $c, \ \varepsilon >0,$
\begin{equation}\label{vol_cond} \frac{V_{i_x}(R)}{V_{i_x}(r)} \geq c \Big(\frac{R}{r}\Big)^{2+ \varepsilon} \quad \text{for all } R > r \geq 1, \end{equation}
then, as in Section 4.4 of \cite{GS5}, we have the estimate
\[ H(x,t) \approx \frac{|x|^2}{V_{i_x}(|x|)}.\]
\end{rem}

\begin{proof}[Proof of Theorem \ref{HK-ends-estimate}:]
Recall $d\mu = \sigma dx.$ Since the restriction to $\Omega$ of any Lipschitz function with compact support in $M$ belongs to $W_0^1(\Omega, \mu)$, in fact $W_0^1(\Omega, \mu) = W^1(\Omega, \mu).$ Hence the Dirichlet forms given by 
\[ \Big(W_0^1(\Omega, \mu), \int_\Omega g(\nabla f, \nabla f) d\mu\Big) \quad \text{and} \quad \Big(W^1(\Omega, \mu), \int_\Omega g(\nabla f, \nabla f) d\mu\Big)\]
coincide. Therefore we can think of $\partial \Omega$ as having no boundary condition, which amounts to considering the heat kernel on the complete manifold (with boundary) $M=\widetilde{\Omega}$, which has Harnack, non-parabolic ends. Hence the result follows from repeating the proofs of Theorems 4.9 and 5.10 in \cite{GS5}. 
\end{proof}

\section{Uniform Domains}
There are several definitions of uniform domains which are equivalent under certain circumstances (see \cite{GyS}, \cite{JLLSC2}, and the references therein). In this section we need only assume we have a length metric space $(M,d),$ that is, a metric space such that $d(x,y)$ is equal to the infimum of the lengths of all continuous curves joining $x$ to $y$ in $M.$ We recall a few definitions as in \cite{GyS}. 

\begin{defin}[Length of a Curve] Let $\gamma : I = [a,b] \mapsto M$ be a continuous curve. Then the length of $\gamma$ is given by 
\[ L(\gamma) = \sup \Big\{ \sum_{i=1}^n d(\gamma(t_{i-1}),\gamma(t_{i})) : n \in \mathbb{N}, a \leq t_0 < \cdots < t_n \leq b\Big\} .\]
\end{defin}

\begin{defin}[Uniform domain]\label{uniform}
Let $U \subset M$ be open and connected. We say $U$ is \emph{uniform} in $M$ if there exist positive, finite constants $c_u, C_U$ such that for any $x,y \in U$ there exists a continuous curve $\gamma_{x,y} : [0,1] \to U$ with $\gamma(0) = x, \gamma(1) = y$ that satisfies 
\begin{enumerate}[(a)]
\item $L(\gamma_{x,y}) \leq C_u d(x,y)$
\item For any $x \in \gamma_{x,y}([0,1]),$
\begin{equation}\label{unifeqn} d(x, \partial U) \geq c_u \frac{L(\gamma_{[x,z]})L(\gamma_{[z,y]})}{L(\gamma_{x,y})},\end{equation}
where for any $z = \gamma_{x,y}(s), \ z'= \gamma_{x,y}(s'), \ 0\leq s \leq s' \leq 1,$ the notation $L(\gamma_ {[z,z']})$ means $L(\gamma|_{[s,s']}).$
\end{enumerate}
\end{defin}

\begin{rem}
A set $U$ satisfying Definition \ref{uniform} is sometimes instead referred to as a \emph{length uniform domain.} In this context, a domain may be called uniform if the length $L$ of curves is replaced by the distance $d$ in $M$ everywhere in (\ref{unifeqn}). However, under a relatively mild condition on balls, these notions are equivalent. (See Theorem 2.7 of \cite{MSar} and Proposition 3.3 of \cite{GyS}, noting that the proof of Proposition 3.3 contains some errors.) In our case of interest, this condition follows from the doubling assumption.
\end{rem}

\begin{rem}
If we replace both the distance $d$ in $M$ in (a) and the lengths of curves $L$ in (\ref{unifeqn}) with $d_U,$ the distance in $U,$ we obtain the definition of an \emph{inner uniform} domain. In the situations considered in the main part of the paper, one can easily check a uniform domain is also inner uniform and all relevant results apply. In fact, in the case where the closure of a set $U$ is nice, $U$ being uniform in its closure is equivalent to $U$ being inner uniform. 
\end{rem}

\section{Green function, parabolic versus non-parabolic}\label{parabolic}
For any weighted Riemannian manifold $\Omega$ with minimal heat kernel $p(t,x,y)$ associated with the Dirichlet form $(W^1_0(\Omega),\int_\Omega g(\nabla f,\nabla f)d\mu)$, we consider  $G(x,y)=\int_0^\infty p(t,x,y)dt$, $x\neq y\in \Omega$. This (extended) function of $x\neq y$ can be identically $+\infty$, in which case we say $\Omega$ is parabolic. If it is finite at some pair $x\neq y,$ then it is finite for all $x\neq y$, and we say that $\Omega$ is 
non-parabolic. In the second case we call $G$ the Green function on $\Omega$. It is a global harmonic function on $\Omega$. There are many characterizations of parabolicity. One of them is that the constant function $\mathbf 1: \Omega\to (0,+\infty)$ is the limit  of a sequence of smooth functions $\phi_n$  with compact support for the norm $(\int_V |f|^2 \, d\mu+\int_V |\nabla f|^2 \, d\mu)^{1/2}$ where $V$ is one (any) fixed non-empty relatively compact open set in $\Omega$.   

Let $(M, \delta M)$ be a complete weighted Riemannian manifold with boundary and assume $M$ has a strictly positive weight $\sigma.$ Then using the above characterization of parabolicity, if $\Omega$ is a submanifold of $M$ such that $M\setminus \Omega$ contains a non-empty hypersurface of codimension $1$,  then  it easily follows that the weighted manifold $\Omega$ is non-parabolic. See \cite{GSurv} for an extensive discussion and references. 

When the weighted Riemannian manifold $\Omega$ is a Harnack weighted manifold,  parabolicity boils down to the volume integral condition
\begin{equation}\label{H-par} \int_1^\infty \frac{ds}{V(x,\sqrt{s})}=+\infty.
\end{equation}
This should be satisfied for one  (equivalently, all) $ x\in \Omega$.  Moreover, when $\Omega$ is  a Harnack weighted manifold that is non-parabolic, its Green function $G$ satisfies
\begin{equation}\label{TSGreen}c_\Omega \int_{d(x,y)^2}^{+\infty} \frac{ds}{V(x,\sqrt{s})}\le G(x,y) \le C_\Omega \int_{d(x,y)^2}^{+\infty} \frac{ds}{V(x,\sqrt{s})}.\end{equation}

\section{Allowing for corners}\label{corners} 

We chose to write our main results in the category of Riemannian manifolds with boundary, but there are no serious difficulties other than notational and expository to apply the same method under various levels of generalization. Because allowing some corners is  very natural in the context of connected sums, we feel compelled to describe briefly a  restrictive but simple set of hypotheses that can  replace the basic assumption that all our manifolds are smooth manifolds with boundary whose metric closures are also smooth manifolds with boundary satisfying $(*)$. 

Let us start with $M^\bullet$, a smooth Riemannian $n$-manifold without boundary and its metric closure $M$.  Let $\Omega$ be an open subset of
$M$ with $M$-topological boundary $\partial \Omega$ contained in $M\setminus M^\bullet$.  In our results up to this point, we were assuming that $M$ was a smooth manifold with boundary, and that $\Omega$ was a manifold with boundary satisfying the extra condition (*). 

Let consider instead the assumption that, for any point $x$ of $M\setminus M^\bullet$, there is a neighborhood $N_x$ of $x$ in $M$,
a Lipschitz map $\Phi_x: \mathbb R^{n-1}\to \mathbb R$ defining 
$R_x=\{(x_1,\dots,x_n): x_n\ge \Phi_x(x_1,\dots,x_{n-1})\}$
and a one-to-one Lipschitz map $\phi_x: N_x\to R_x$  which is bi-Lipschitz on its image $V_x$. The Lipschitz constants associated to $\Phi_x$ and $\phi_x, \ \phi^{-1}_x$ may depend on $x$ but are locally bounded on $M\setminus M^\bullet$. The (minimal) heat kernel on the (weighted) smooth manifold $(M^\bullet,\mu)$ is well defined as usual.
The (``Neumann type") heat kernel on $(M,\mu)$ is also easily defined, being associated with the  regular strictly local Dirichlet form $\int_M|\nabla f|^2d\mu$ with domain the set of all functions  $f$ in $W_{\mbox{\tiny loc}}(M^\bullet)$ such that $\int_M (|f|^2+|\nabla f|^2) d\mu<+\infty.$ In this Dirichlet space whose underlying space is $M$, solutions of the heat equation satisfy the local parabolic Harnack inequality. Any open set $\Omega$
with $\partial \Omega\subseteq M\setminus M^\bullet$ is locally inner-uniform in $M$ (see \cite[Section 3.2]{JLLSC2} for details on local inner-uniformity). By \cite{JLLSC2,JLLSC1}, it follows that  harmonic functions in $\Omega$ which vanish on $\partial \Omega$ satisfy the local version of the boundary elliptic Harnack inequality. Together, these facts allow for the generalization of the results of this paper in this context. The key difference lies in the way in which positive harmonic functions vanish at the boundary. On smooth manifolds with boundary, positive harmonic functions vanishing at the boundary vanish linearly. In the more general context described above, the best one can say is already contained in the boundary Harnack inequality, and vanishing of the type $d(x,\partial \Omega)^{\eta_{x_0}}$ when $x$ tends to $x_0\in \partial \Omega,$ with $\eta_{x_0}\in [0,1]$, is typical. Without entering into all the details necessary to make the above line of reasoning precise, it can easily be implemented to cover the very basic examples with corners shown in Figures
\ref{fig:cones1}--\ref{fig:cones3} of the introduction.  

 \section{Connection with earlier results}\label{earlier_results} To help the reader understand the techniques and estimates discussed above and relate them to the existing literature, we illustrate how they include some known results. Even though our focus is on manifolds with finitely many nice ends, this section discusses the simpler case when there is only one end. 
  
Consider a complete Harnack Riemannian manifold $M$ (e.g., a complete manifold with non-negative Ricci curvature) and a domain $\Omega=M\setminus K$. When $K$ is a bounded $\mathcal C^{1,1}$ domain and Dirichlet condition is assumed on the boundary of $\Omega$ (assume for simplicity that $\Omega$ is a domain, hence connected), \cite{Zglob} gives  global two-sided heat kernel estimates for $p_\Omega(t,x,y)$ at all times and locations. 
In the case where $M$ is non-parabolic, the estimates of \cite[Theorem 1.1(a)]{Zglob} compare 
$  p_\Omega(t,x,y)$ (at all times $t$ and locations $x,y$) to expressions of the form 
$$C\left(\frac{d(x, K)}{\sqrt{t}\wedge1}\wedge 1\right)\left(\frac{d(y,K)}{\sqrt{t}\wedge 1}\wedge 1\right)\frac{\exp\left(-c
\frac{d(x,y)^2}{t}\right)}{V(x,\sqrt{t})}.$$
The key ingredients in \cite{Zglob} are
(a) near boundary estimates based on  \cite{FGS} and the $\mathcal C^{1,1}$ nature of the boundary (see also \cite{Zloc}) and (b) global estimates away from the boundary from \cite{GS3} treating the case when $d(x,K)$ and $d(y,K)$ are greater than $1$.  

The validity of such two sided global heat kernel estimates are extended in several different directions in \cite[Theorem 5.11]{GyS}.  Theorem 5.11 of \cite{GyS} applies 
to a domain $\Omega=M\setminus K$ in a complete Harnack manifold whenever $\Omega$ is uniform (in fact, inner-uniform suffices--see, e.g., \cite{GyS,JLLSC2}).  In such cases (and without the hypothesis of non-parabolicity), there exists
a positive harmonic function $h$ on $\Omega$, vanishing appropriately when reaching $K$, such that the heat kernel
$p_\Omega(t,x,y)$ compares (at all times and locations $x,y$) to expressions of the form
$$C\frac{h(x)h(y)}{V_{h}(x,\sqrt{t})}\exp\left(-c
\frac{d(x,y)^2}{t}\right).$$
Here $V_{h}(x,r)=\int_{B(x,r)}h^2(z)dz\approx h(x_r)^2V(x,r)$ where $x_r$ is a point at distance at most $Ar$ from $x$ and at least $ar$ from $K$ for some appropriate fixed constants $a,A$ (see Theorem \ref{HarnackEnds}). 
When $\Omega =M\setminus K$ is connected and $K$ is a bounded $\mathcal C^{1,1}$ domain, $\Omega$ is automatically uniform and, by classical theory, $h$ vanishes linearly near the boundary of $\Omega$.  If, in addition,
$M$ is non-parabolic then $h(x) \approx  d(x,K)\wedge 1$ and, by simple computation, one recovers the estimates of \cite{Zglob}.   
This yields a different proof of the results in \cite{Zglob}, independent from the earlier results in \cite{FGS,GS3}. 

In addition,  \cite[Theorem 5.11]{GyS} allows for $K$ to be unbounded and non-smooth as long as the key hypothesis that $\Omega\setminus K$ is uniform (in fact, inner-uniform) remains. In fact, because it is stated in the setting of Dirichlet spaces,  \cite[Theorem 5.11]{GyS} allows for the treatment of mixed boundary condition. For instance, as in Example \ref{paraboloid_ex}, 
one can take $M=\mathbb R^d$, $K=\{x=(x_1,\dots, x_d): x_1^2+\dots x_{d-1}^2\le x_d\}$ (a paraboloid of revolution),
 and $\Omega=M\setminus K$. Moreover, one can impose mixed boundary condition 
 along  $\partial K=\partial \Omega$.   The technique of \cite{GyS} is to obtain estimates on $p_\Omega(t,x,y)$ via intermediate heat kernel estimates on a related heat kernel, the heat kernel $p_{\Omega,h^2}(t,x,y)$ of the weighted manifold $(\Omega, h^2)$ where $h$ is the harmonic profile of the domain $\Omega$. The key point is that when $\Omega$ is uniform, one can prove that the profile $h$ has good properties that imply $(\Omega, h^2)$ is a Harnack manifold. This implies that classical two-sided heat kernel bounds hold for $p_{\Omega,h^2}(t,x,y)$ (see Theorem  \ref{HarnackEnds} above). The estimates for $p_\Omega(t,x,y)$ then follow simply because $p_\Omega(t,x,y)=h(x)h(y)p_{\Omega,h^2}(t,x,y)$.  One important aspect of this approach  is that it resolves all at once the problems related to the global geometric structure of the manifold  $M$ and domain $\Omega$, and the local problems related to the presence of boundary conditions.
 
The strategy from \cite{GyS} explained above is implemented in this paper to prove our main result, Theorem \ref{main_thm}, using the function $h$ constructed in Theorem \ref{thm-profile}, and Theorem \ref{HK-ends-estimate} applied to the weighted manifold $(M,\sigma h^2)$.

\end{appendices}

\end{document}